\documentclass[11pt]{article}

\usepackage{amsfonts}
\usepackage{amsmath}
\usepackage{amsthm}
\usepackage{amssymb}
\newcommand{\subjclass}{2000AMS {\it Subject Classification }}
\newcommand{\br}{\mathbb R}
\newcommand{\bn}{\mathbb N}
\newcommand{\tow}{\rightharpoonup}
\numberwithin{equation}{section}
\newtheorem{theorem}{Theorem}[section]
\newtheorem{lemma}[theorem]{Lemma}
\newtheorem{corollary}[theorem]{Corollary}
\newtheorem{proposition}[theorem]{Proposition}

\title{ Existence,  Uniqueness,  and Convergence of optimal control problems
 associated with Parabolic variational inequalities of the second kind}
\author{\it \bf{Mahdi Boukrouche}\thanks{Lyon University, F-42023 Saint-Etienne,
Laboratory of Mathematics, University of
 Saint-Etienne, LaMUSE EA-3989,
23 Docteur Paul Michelon
42023 Saint-Etienne Cedex 2, France. Fax: +33 4 77 48  51 53, Phone +33 4 77 48 15 35,
  E-mail: Mahdi.Boukrouche@univ-st-etienne.fr.} \and
{\bf Domingo A. Tarzia}\thanks{Departamento de Matem\'atica-CONICET, FCE,
Univ. Austral, Paraguay 1950, S2000FZF Rosario, Argentina. Fax: +54 341 522 3001, Phone +54 341 522 3093,
E-mail: DTarzia@austral.edu.ar}
}
\date{}
\begin{document}
\maketitle
\normalsize
\begin{abstract}
Let $u_{g}$  the unique solution of a parabolic variational  inequality  of  second kind,
with a given  $g$.
Using a regularization method, we prove,  for all  $g_{1}$ and $g_{2}$,   a monotony property  between
$\mu u_{g_{1}} + (1-\mu)u_{g_{2}}$  and $u_{\mu g_{1} + (1-\mu)g_{2}}$ for $\mu \in [0 , 1]$.
This  allowed us to prove the existence and uniqueness results to
a family of optimal control  problems over  $g$ for each heat transfer coefficient  $h>0$, associated to the  Newton law, and of another  optimal control problem associated to a Dirichlet boundary  condition.
 We prove also , when $h\to +\infty$, the strong  convergence
of the  optimal controls and states associated to this  family of optimal control problems with the  Newton law
 to that of the optimal control problem associated to a Dirichlet boundary  condition.



\smallskip
{\bf Short title :} Controls for parabolic variational inequalities
\end{abstract}

\section{Introduction}\label{intro}
Let  consider the following problem governed by the parabolic variational inequality
\begin{equation}\label{eq1}
\langle \dot{u}(t) \,, \, v-u(t)\rangle + a(u(t) \, ,\, v-u(t)) + \Phi(v) - \Phi(u(t)) \geq  <g(t) \,,\, v- u(t) >  \quad
\forall v\in K,
\end{equation}
a.e. $t\in ]0 , T[$, with the initial condition
\begin{equation}\label{ic1}
u(0)= u_{b},
\end{equation}
where,  $a$ is  a symmetric continuous and coercive bilinear form on the Hilbert space $V\times V$, $\Phi$  is a proper and convex function from $V$ into $\mathbb{R}$ and is lower semi-continuous for the weak topology on $V$,
$< \cdot , \cdot>$ denotes the  duality brackets between $V'$ and $V$,
 $K$ is  a closed convex non-empty  subset of $V$, $u_{b}$ is an initial value in another Hilbert space $H$ with $V$ being  densely and continuously  imbedded in $H$, and  $g$ is a given function in the space $L^{2}(0 , T , V')$.
It is well known \cite{Brezis68, Brezis72, Chipot2000, DL1972} that, there exists  a unique solution
$$u\in {\cal C}(0 , T , H)\cap L^{2}(0 , T , V)  \quad {\rm with }
\quad \dot{u}= \dfrac{\partial u}{\partial t}\in L^{2}(0 , T , H)
$$
 to (\ref{eq1})-(\ref{ic1}). So we can consider $g \mapsto u_{g}$  as a function from
$L^{2}(0 , T , H)$ to ${\cal C}(0 , T , H)\cap L^{2}(0 , T , V)$. Then we can consider \cite{KeMu2008, JLL, NST2006}
 the cost functional $J$ defined by
\begin{eqnarray}\label{e4.1}
 J(g)= {1\over 2}\|u_{g}\|^{2}_{L^{2}(0 , T , H)} + {M\over 2}\|g\|^{2}_{L^{2}(0 , T , H)},
\end{eqnarray}
where $M$ is a positive constant, and $u_{g}$ is the unique solution to (\ref{eq1})-(\ref{ic1}), corresponding to the control $g$.
One of our  main purposes is to prove the existence and uniqueness of the optimal control problem
\begin{equation}\label{P}
 \mbox{Find }  g_{op}\in L^{2}(0 , T, H) \quad\mbox{such that}
\quad J(g_{op})= \min_{g\in L^{2}(0 , T, H)} J(g).
\end{equation}
This can be reached if we prove  the strictly convexity of the cost functional $J$, which  follows
(see Theorem \ref{th4.2}) from
 the following  monotony property : {\it for any two control $g_{1}$ and $g_{2}$ in $L^{2}(0 , T , H)$,  }
\begin{eqnarray}\label{mopr}
u_{4}(\mu)\leq u_{3}(\mu)   \qquad \forall \mu\in [0 , 1],
\qquad\label{u3}
\end{eqnarray}
where
\begin{eqnarray}\label{u34}
u_{3}(\mu)= \mu u_{1}+ (1-\mu)u_{2},  \qquad u_{4}(\mu)= u_{g_{3}(\mu)}, \quad {\rm with} \quad g_{3}(\mu)= \mu g_{1} + (1-\mu)g_{2}.
\end{eqnarray}
In Section {\rm\ref{sec-3}}, we   establish first in Theorem \ref{th1}, the error
estimate between  $u_{3}(\mu)$ and $u_{4}(\mu)$.
This result   generalizes our previous result obtained in \cite{MB-DT1} for the elliptic variational inequalities.
 We deduce  in  Corollary \ref{cor22}  a condition on the data to get
                      $u_{3}(\mu)= u_{4}(\mu)$ for all $\mu\in [0 , 1]$.
Then we assume,   that the convex  $K$ is a subset of $V=H^{1}(\Omega)$ and consider the parabolic variational problems ($P$) and  ($P_{h}$).
So, using a regularization method, we prove  in Theorem \ref{la} this  monotony property {\rm(\ref{mopr})},   for the solutions of the two problems
 ($P$) and  ($P_{h}$).
 This result with a new   proof and simplified, generalizes  that  obtained by \cite{Mingot1} for elliptic variational inequalities.
In  Subsection \ref{sub-3.2} we also obtain some  properties of dependency solutions based on the data
$g$ and on a positive parameter $h$ for the parabolic variational inequalities   (\ref{eq1}) and  (\ref{iv2}),
  see Propositions \ref{l2.3}, \ref{l2.3h} and \ref{l2.3h2}.

In Section \ref{secOPC},  we consider  the family of distributed optimal control problems $(P_{h})_{h>0}$,
\begin{equation}\label{Ph}
 \mbox{Find }  g_{op_{h}}\in L^{2}(0 , T , H) \quad \mbox{such that }
    \quad J(g_{op_{h}})= \min_{g\in L^{2}(0 , T , H)} J_{h}(g),
\end{equation}
 with the cost functional
\begin{equation}\label{Jh}
 J_{h}(g)= {1\over 2}\|u_{g_{h}} \|_{L^{2}(0 , T , H)}^{2}+ {M\over 2}\|g\|_{L^{2}(0 , T , H)}^{2},
\end{equation}
where $u_{g_{h}}$ is the unique solution of  (\ref{iv2})-(\ref{ic1}), corresponding to the control $g$ for each $h>0$,
and  the distributed optimal control problems
\begin{equation}\label{P}
 \mbox{Find }  g_{op}\in L^{2}(0 , T , H) \quad \mbox{such that }
    \quad J(g_{op})= \min_{g\in L^{2}(0 , T , H)} J(g),
\end{equation}
 with the cost functional (\ref{e4.1}) where
$u_{g}$ is the unique solution to (\ref{eq1})-(\ref{ic1}), corresponding to the control $g$.
Using  Theorem \ref{la} with its crucial property of monotony (\ref{u3}),
we prove the strict convexity  of  the cost functional  (\ref{e4.1}) and also of the cost functional (\ref{Jh}),
associated to the  problems (\ref{P}) and (\ref{Ph}) respectively.
Then, the existence and uniqueness of solutions to the optimal controls problems (\ref{P}) and (\ref{Ph}) follows from \cite{JLL}.

In general see for example \cite{CaJa1959}
 the  relevant  physical condition, to impose on the boundary,   is Newton's  law,  or  Robin's law,  and not Dirichlet's.
Therefore, the objective of this work is to approximate the optimal control problem (\ref{P}), where the state is the solution to parabolic
variational problem  (\ref{eq1})-(\ref{ic1}) associated with the Dirichlet condition (\ref{pbc1}),
 by a family indexed by a factor $h$ of optimal control problems (\ref{iv2})-(\ref{ic1}),
where states are the solutions to parabolic variational problems,
associated with the boundary condition of Newton  (\ref{pbc3}).
 Moreover, from a numerical analysis point of view
it maybe preferable to consider approximating Neumann problems in
all space $V$ (see (\ref{iv2})-(\ref{ic1})), with
parameter $h$, rather than the Dirichlet problem in a subset of
the space $V$ (see (\ref{eq1})-(\ref{ic1})).
So the asymptotic behavior can be considered very important in the optimal control.

In the last subsection \ref{secLim},  which is also the goal of our paper,  we prove  that the optimal control $g_{op_{h}}$
 (unique solution of the optimization problem (\ref{Ph}))
 and its corresponding state  $u_{g_{op_{h}h}}$ (the unique solution of the parabolic variational problem
(\ref{iv2})-(\ref{ic1}))  for each $h>1$, are strongly convergent  to $g_{op}$ (the unique solution of the optimization
 problem (\ref{P})), and $u_{g_{op}}$ (the unique solution of the parabolic variational problem (\ref{eq1})-(\ref{ic1}))
 in $L^{2}([0 , T]\times\Omega)$ and $L^{2}(0 , T , H^{1}(\Omega))$ respectively when $h\to +\infty$.

 This paper  generalizes the results obtained in \cite{GT}, for  elliptic variational equalities,   and
 in \cite{MT} for parabolic variational  equalities, to  the case of parabolic variational inequalities of second kind.
Various problems with distributed optimal control,  associated with   elliptic variational inequalities are given see for example
\cite{A2006, Barbu84a}, \cite{B1997a}-\cite{BM2000}, \cite{Ca2000,  IK2000},
 \cite{Mingot1}-\cite{NPS2006}, \cite{YC2004} and for the parabolic case see for example
\cite{Amassad2008, Barbu84a, Barbu84c}, \cite{BerMer2000}-\cite{Bonas1986}, \cite{NT88, NT},
\cite{pironneau1984}.

\section{On  the  property of monotony}\label{sec-3}
As we can not prove the  property of monotony {\rm(\ref{mopr})}  for any convex set $K$.
Let $\Omega$  a bounded open set in $\br^{N}$ with smooth boundary  $\partial\Omega=\Gamma_{1}\cup\Gamma_{2}$.
 We assume  that   $\Gamma_{1}\cap\Gamma_{2}=\O$, and  $meas(\Gamma_{1})>0$. Let $H= L^{2}(\Omega)$, $V= H^{1}(\Omega)$.
We can prove  the  property of monotony {\rm(\ref{mopr})} for any convex subset of $V$.
Let
$$K=\{v\in V : \quad v_{|\Gamma_{1}}=0\},  \quad and \quad  K_{b}=\{v\in V : \quad v_{|\Gamma_{1}}= b\}. $$
So we consider the  following variational  problems with such convex  subset.

\noindent{\bf Problem ($P$)}
Let given  $b\in L^{2}(]0 , T[\times\Gamma_{1})$,  $g\in L^{2}(0 , T , H)$ and $q\in  L^{2}(]0 , T[\times \Gamma_{2})$,  $q>0$.
 Find    $u$ in  ${\cal C}([0 , T] , H)\cap  L^{2}(0 , T, K_{b})$ solution of
 the  parabolic problem (\ref{eq1}), where  $< \cdot , \cdot >$
is only the scalar product $( \cdot , \cdot )$ in $H$,  with the initial condition  (\ref{ic1}),
and $\Phi(v)=\int_{\Gamma_{2}} q|v| ds.$

\smallskip
\noindent{\bf Problem ($P_{h}$)}
Let given  $b\in L^{2}(]0 , T[\times\Gamma_{1})$,  $g\in L^{2}(0 , T , H)$  and $q\in  L^{2}(]0 , T[\times \Gamma_{2})$, $q>0$.
For all  coefficient $h>0$, find  $u\in {\cal C}(0 , T, H)\cap L^{2}(0 , T , V)$ solution of
  the  parabolic variational inequality
\begin{eqnarray}\label{iv2}
\langle \dot{u}(t) \, , \,  v-u(t)\rangle +  a_{h}(u(t) \, ,\,  v-u(t))
+ \Phi(v)- \Phi(u(t))\geq
( g(t) , v-u(t)) \nonumber\\
+ h\int_{\Gamma_{1}}b(t)(v-u(t))ds\quad \forall v\in V,
\end{eqnarray}
and the initial condition  {\rm(\ref{ic1})},
where
$
a_{h}(u , v)= a(u , v)+ h\int_{\Gamma_{1}}u  v ds.
$

It is easy to see that the problem ($P$) is with the Dirichlet condition
\begin{eqnarray}
  u= b \quad on \quad\Gamma_{1}\times]0 , T[, \label{pbc1}
\end{eqnarray}
and the problem ($P_{h}$) is with the following Newton-Robin's type  condition
\begin{eqnarray}
 -\dfrac{\partial u}{ \partial n}= h(u-b)\quad on \quad
    \Gamma_{1}\times]0 , T[. \label{pbc3}
\end{eqnarray}
 where $n$ is the exterior unit vector normal to the boundary. The integal on $\Gamma_{2}$ in the expression
of  $\Phi$ comes from  the Tresca  boundary condition (see
 \cite{mb-Ls1}-\cite{sf1},\cite{DL1972})  with   $q$ is  the Tresca friction coefficient  on $\Gamma_{2}$.
Note that only for the proof of  Theorem \ref{la}   we have need to specify an expression of the functional  $\Phi$.

By assumption there exists  $\lambda >0$ such that
$
\lambda \|v\|_{V}^{2}\leq a( v \, , \, v)\quad \forall v\in V$.
Moreover, it follows from  \cite{DT, Ta1979} that there exists $\lambda_{1}>0$ such that
$$ a_{h}(v , v)\geq \lambda_{h}\|v\|_{V}^{2} \quad \forall v\in V,
\quad \mbox{ with }\lambda_{h}= \lambda_{1}\min\{1  \,, \, h\}
$$
so  $a_{h}$ is a bilinear,  continuous, symmetric and coercive form on $V$. So there exists an unique solution to each of the two problems
 ($P$) and ($P_{h}$).

We recall  that $u_{g}$ is the unique solution of the parabolic variational problem ($P$),
corresponding to the control  $g\in L^{2}(0 , T , H)$, and also that  $u_{g_{h}}$ is
the unique solution of the  parabolic variational problem ($P_{h}$),
corresponding to the control  $g\in L^{2}(0 , T , H)$.

\begin{proposition}\label{r2.3.2}
Assume that  $g\geq 0$ in $\Omega\times ]0 , T[$, $b\geq 0$ on $\Gamma_{1}\times ]0 , T[$,
 $u_{b}\geq 0$ in $\Omega$.
Then as $q>0$, we have $u_{g}\geq 0$. Assuming again that  $h>0$, then  $u_{g_{h}}\geq 0$ in $\Omega\times ]0 , T[$.
\end{proposition}
 \begin{proof} For $u=u_{g_{h}}$, it is enough to take $v=u^{+}$   in {\rm(\ref{iv2})},    to get
\begin{eqnarray}\label{pmx}
\|u^{-}(T)\|^{2}_{L^{2}(\Omega)} +  \lambda \int_{0}^{T} \|u^{-}(t)\|_{V}^{2} dt
+  h\int_{0}^{T}\int_{\Gamma_{1}}(u^{-}(t))^{2}ds  dt
+     \leq - \int_{0}^{T}( g(t) , u^{-}(t)) dt  \nonumber\\
 -\int_{0}^{T}\int_{\Gamma_{2}}q (|u(t)|-|u^{+}(t)|)dsdt
-h\int_{0}^{T}\int_{\Gamma_{1}}b(t) u^{-}(t) ds dt + \|u^{-}(0)\|^{2}_{L^{2}(\Omega)} \quad
\end{eqnarray}
so the result follows. 
\end{proof}

\begin{theorem}\label{th1}
Let $u_{1}$ and $u_{2}$ be two  solutions of the parabolic variational inequality  {\rm(\ref{eq1})}
with the same initial condition,   and corresponding to the two control   $g_{1}$ and $g_{2}$ respectively. We have the following estimate
\begin{eqnarray*}\label{eq3.1}
 {1\over 2}\|u_{4}(\mu) -u_{3}(\mu)\|^{2}_{L^{\infty}(0 , T, H)}
 +   \lambda \|u_{4}(\mu) -u_{3}(\mu)\|^{2}_{L^{2}(0 , T, V)}
+ \mu {\cal I}_{14}(\mu)(T) + (1-\mu) {\cal I}_{24}(\mu)(T)
\nonumber\\
+\mu\Phi(u_{1}) + (1-\mu)\Phi(u_{2})- \Phi(u_{3}(\mu))
\leq \mu(1-\mu)({\cal A}(T , g_{1}) + {\cal B}(T, g_{2})) \quad \forall \mu\in [0  , 1],
\end{eqnarray*}
where
$${\cal  I}_{j4}(\mu)(T) =  \int_{0}^{T} I_{j4}(\mu)(t) dt \quad{\rm for \,} j=1 , 2,
\quad
{\cal A}(T, g_{1}) = \int_{0}^{T}\alpha(t) dt, \quad {\cal B}(T, g_{2}) = \int_{0}^{T}\beta (t) dt,$$
\begin{eqnarray*}
I_{j4}(\mu)= \langle  \dot{u}_{j} \,, \, u_{4}(\mu)  -u_{j} \rangle + a(u_{j}\, , \, u_{4}(\mu) - u_{j}) + \Phi(u_{4}(\mu)) -\Phi(u_{j})
-\langle g_{j} , u_{4}(\mu) - u_{j}\rangle \geq 0,
\end{eqnarray*}
\begin{gather}\label{alp}
\alpha = \langle \dot{u}_{1} \,, \, u_{2}  -u_{1} \rangle+ a(u_{1}\, , \, u_{2} - u_{1} )+ \Phi(u_{2}) -\Phi(u_{1})
-\langle g_{1} , u_{2} - u_{1}\rangle  \geq 0, \\
\label{bet}
\beta =  \langle   \dot{u}_{2} \,, \, u_{1}  -u_{2}\rangle + a(u_{2}\, , \, u_{1} - u_{2} )+ \Phi(u_{1}) -\Phi(u_{2})
 -\langle g_{2} , u_{1} - u_{2}\rangle \geq 0.
\end{gather}
\end{theorem}
\begin{proof}
   As $u_{3}(\mu)(t) \in K$ so with $v= u_{3}(\mu)(t)$, in the variational inequality (\ref{eq1}) where $u=u_{4}(\mu)$ and $g=g_{3}(\mu)$,   we obtain
\begin{eqnarray*}
 \langle  \dot{u}_{4}(\mu) \,, \, u_{3}(\mu) -u_{4}(\mu) \rangle +      a(u_{4}(\mu) \, , \, u_{3}(\mu)-u_{4}(\mu))
 + \Phi(u_{3}(\mu)) -\Phi(u_{4}(\mu))\nonumber\\
\geq\langle g_{3}(\mu) , u_{3}(\mu)- u_{4}(\mu)\rangle \quad  a.e. \, t\in ]0 , T[,
\end{eqnarray*}
then
\begin{eqnarray*}
 \langle  \dot{u}_{4}(\mu) - \dot{u}_{3}(\mu) \,, \, u_{4}(\mu) -u_{3}(\mu) \rangle + a(u_{4}(\mu)- u_{3}(\mu) \, , \, u_{4}(\mu)-u_{3}(\mu))  \qquad
\nonumber\\
\leq  \langle  \dot{u}_{3}(\mu) \,, \, u_{3}(\mu) -u_{4}(\mu) \rangle +
 a(u_{3}(\mu) \, , \, u_{3}(\mu)-u_{4}(\mu))\nonumber\\
+ \Phi(u_{3}(\mu)) -\Phi(u_{4}(\mu)(t))
 -\langle g_{3}(\mu) , u_{3}(\mu)- u_{4}(\mu)\rangle \quad  a.e. \, t\in ]0 , T[,
\end{eqnarray*}
thus
\begin{eqnarray*}
 {1\over 2}{\partial \over\partial t}\left(\|u_{4}(\mu) -u_{3}(\mu)\|^{2}_{H}\right) +    \lambda\|u_{4}(\mu)- u_{3}(\mu)\|_{V}^{2}
\leq
\langle  \dot{u}_{3}(\mu) \,, \, u_{3}(\mu) -u_{4}(\mu) \rangle \nonumber\\
+ a(u_{3}(\mu) \, , \, u_{3}(\mu)-u_{4}(\mu))+ \Phi(u_{3}(\mu)) -\Phi(u_{4}(\mu))
\nonumber\\-
\langle g_{3}(\mu) \, ,\,  u_{3}(\mu)- u_{4}(\mu)\rangle, \quad  a.e. \, t\in ]0 , T[,
\end{eqnarray*}
using that $u_{3}(\mu)=\mu(u_{1} -u_{2})+  u_{2}$, $g_{3}(\mu)=\mu(g_{1} -g_{2})+  g_{2}$
we get
\begin{eqnarray*}
 {1\over 2}{\partial \over\partial t}\left(\|u_{4}(\mu) -u_{3}(\mu)\|^{2}_{H}\right) +    \lambda\|u_{4}(\mu)- u_{3}(\mu)\|_{V}^{2} +\mu\Phi(u_{1}) + (1-\mu)\Phi(u_{2})- \Phi(u_{3}(\mu)
 \nonumber\\ \leq
\mu(1-\mu)(\alpha + \beta)- \mu I_{1 4}(\mu) - (1-\mu) I_{2 4}(\mu) \quad  a.e. \, t\in ]0 , T[,
\end{eqnarray*}
so by integration  between  $t=0$ and $t=T$, we deduce the required result.
\end{proof}

\begin{corollary}\label{cor22} From Theorem \rm{\ref{th1}} we get $a.e.  \, t\in [0 , T]$
\begin{eqnarray*}
{\cal A}(T,g_{1})={\cal B}(T , g_{2})= 0 \Rightarrow
\left\{
\begin{array}{ll}
&u_{3}(\mu) =u_{4}(\mu)  \qquad \forall \mu\in [0 , 1],
\\
&I_{14}(\mu)=  I_{24}(\mu)=0 \qquad \forall \mu\in [0 , 1],
\\
& \Phi(u_{3}(\mu))=\mu\Phi(u_{1}) + (1-\mu)\Phi(u_{2})\qquad \forall \mu\in [0 , 1]. 
 \end{array}
\right.
\end{eqnarray*}
\end{corollary}

\begin{lemma}\label{l1}
Let $u_{1}$ and $u_{2}$ be two  solutions of the parabolic  variational inequality of second kind {\rm(\ref{eq1})}
with respectively  as second member $g_{1}$ and $g_{2}$, then we get
\begin{equation}\label{3.4}
\|u_{1} -u_{2}\|^{2}_{L^{\infty}(0 , T, H)} +
 \lambda\|u_{1} -u_{2}\|^{2}_{L^{2}(0 , T, V)}
\leq {1\over \lambda }\|g_{1}-g_{2}\|^{2}_{L^{2}(0 , T, V')},
\end{equation}
Where $\lambda$ is the coerciveness constant of the biliear form $a$.
 \end{lemma}
\begin{proof}
Taking $v=u_{2}$ in (\ref{eq1}) where $u=u_{1}$ and
$g= g_{1}$;  then   $v=u_{1}$ in (\ref{eq1}) where $u=u_{2}$
and $g= g_{2}$, so by addition 
(\ref{3.4}) holds.
\end{proof}

We generalize now  in our case the result on a monotony property, obtained by \cite{Mingot1}
for the elliptic variational inequality. This theorem is the  cornestone to prove the strict convexity
of the cost functional $J$ defined in Problem (\ref{P}) and  the cost functional $J_{h}$ defined in Problem (\ref{Ph}).
Remark first that with  the  duality bracks  $< \cdot , \cdot >$   defined by
$$ < g(t)  , \varphi >= (g(t)  , \varphi) + h\int_{\Gamma_{1}} b(t) \varphi ds$$
 (\ref{iv2})  leads to ({\ref{eq1}).
We prove the following theorem for  $\Phi$ such that $\Phi(v)=\int_{\Gamma_{2}} q|v| ds$.
\begin{theorem}\label{la}
For any two control $g_{1}$ and $g_{2}$ in $L^{2}(0, T, H)$,  it holds that
\begin{equation}\label{eq3.4}
 u_{4}(\mu) \leq u_{3}(\mu)\quad in \quad \Omega\times[0 , T],
    \quad \forall \mu\in [0 , 1].
\end{equation}
Here  $u_{4}(\mu)= u_{\mu g_{1}+(1-\mu) g_{2}}$,  $u_{3}(\mu)= \mu u_{g_{1}}+(1-\mu) u_{g_{2}}$,
$u_{1}=u_{g_{1}}$ and $u_{2}=u_{g_{2}}$ are the unique solutions of the variational  problem $P$,
with $g=g_{1}$ and $g=g_{2}$ respectively, and
  for the same $q$, and the same initial condition  {\rm(\ref{ic1})}.
Moreover,   it holds also that
\begin{equation}\label{eq3.4h}
 u_{h4}(\mu) \leq u_{h3}(\mu)\quad in \quad \Omega\times[0 , T],
    \quad \forall \mu\in [0 , 1].
\end{equation}
Here  $u_{4h}(\mu)= u_{\mu g_{1h}+(1-\mu) g_{2h}}$,  $u_{3h}(\mu)= \mu u_{g_{1h}}+(1-\mu) u_{g_{2h}}$,
$u_{1h}=u_{g_{1h}}$ and $u_{h2}=u_{g_{h2}}$ are the unique solutions of the variational  problem {\rm$P_{h}$},
 with $g=g_{1}$ and $g=g_{2}$ respectively, and
  for the same $q$, $h$, $b$ and the same initial condition  \rm{(\ref{ic1})}.
\end{theorem}
\begin{proof}
The main difficulty, to prove this result comes from the fact that  the functional $\Phi$ is not differentiable.
To overcome this difficulty, we use the regularization  method and consider for $\varepsilon >0$ the following approach of $\Phi$
$$\Phi_{\varepsilon}(v)= \int_{\Gamma_{2}}q\sqrt{\varepsilon^{2} + |v|^{2}} ds, \qquad \forall v\in V,$$
which is Gateaux differentiable, with
$$\langle \Phi'_{\varepsilon}(w) \,, \, v\rangle =  \int_{\Gamma_{2}}
{qw v \over \sqrt{\varepsilon^{2} +|w|^{2}}} ds \qquad \forall (w , v)\in V^{2}.
$$
Let $u^{\varepsilon}$ be the unique solution of the variational inequality
\begin{eqnarray}\label{ive}
\langle \dot{u}^{\varepsilon}\, , \, v- u^{\varepsilon} \rangle + a(u^{\varepsilon}\, , \, v- u^{\varepsilon})
  + \langle \Phi'_{\varepsilon}(u^{\varepsilon}) \,, \, v-u^{\varepsilon} \rangle
 \geq  \langle g \,, \, v-u^{\varepsilon} \rangle  \quad a.e.  \, t\in [0 , T]
\nonumber\\
\forall v\in K, \mbox{ and } u^{\varepsilon}(0)= u_{b}.
\end{eqnarray}
Let us show first that for all  $\mu\in [0 , 1]$ $u^{\varepsilon}_{4}(\mu)\leq u^{\varepsilon}_{3}(\mu)$,
then that $u^{\varepsilon}_{3}(\mu) \to u_{3}(\mu)$ and   $u^{\varepsilon}_{4}(\mu) \to u_{4}(\mu)$ strongly in $L^{2}(0 , T; H)$ when $\varepsilon\to 0$.
Indeed for all $\mu\in [0 , 1]$, let consider $U_{\varepsilon}(\mu)=u^{\varepsilon}_{4}(\mu) - u^{\varepsilon}_{3}(\mu)$ thus  $u^{\varepsilon}_{4}(\mu)(t) -U_{\varepsilon}^{+}(\mu)(t)$ is in $K$. So we can take $v= u^{\varepsilon}_{4}(\mu)(t) -U_{\varepsilon}^{+}(\mu)(t)$ in (\ref{ive}) where $u^{\varepsilon}=u^{\varepsilon}_{4}(\mu)$ and $g=g_{3}(\mu)=\mu (g_{1}-g_{2})+g_{2}$.
We also can take  $v=u^{\varepsilon}_{1}(t) +U_{\varepsilon}^{+}(\mu)(t)$ in (\ref{ive}) where $u^{\varepsilon}=u^{\varepsilon}_{1}$ and $g=g_{1}$, and we multiply the two sides of the obtained inequality by $\mu$
then we take $v=u^{\varepsilon}_{2} +U_{\varepsilon}^{+}(\mu)$ in (\ref{ive}) where $u^{\varepsilon}=u^{\varepsilon}_{2}$ and $g=g_{2}$ and we multiply the two sides of the obtained inequality by $(1-\mu)$.
By adding the three  obtained inequalities  we get $a.e.  \, t\in ]0 , T[$,
\begin{eqnarray*}
{1\over 2}{\partial \over\partial t}(\|U_{\varepsilon}^{+}(\mu)\|_{H}^{2})
+ \lambda \|U_{\varepsilon}^{+}(\mu)\|_{V}^{2}\leq
\langle \mu \Phi'_{\varepsilon}(u^{\varepsilon}_{1})
         + (1-\mu) \Phi'_{\varepsilon}(u^{\varepsilon}_{2}) -\Phi'_{\varepsilon}(u^{\varepsilon}_{4}(\mu)) \,, \, U^{+}_{\varepsilon} (\mu)\rangle,
\end{eqnarray*}
hence as $U_{\varepsilon}^{+}(\mu)(0)=0$,   by integration from $t=0$ to $t=T$ we obtain  a.e.  $t\in ]0 , T[$
\begin{eqnarray*}
{1\over 2}\|U_{\varepsilon}^{+}(\mu)(T)\|^{2}_{H}+ \lambda\int_{0}^{T}
\| U_{\varepsilon}^{+}(\mu)(t)\|^{2}_{V} dt  \leq\qquad  \qquad
\nonumber\\
\leq
\int_{0}^{T}\langle \mu \Phi'_{\varepsilon}(u^{\varepsilon}_{1}(t))
         + (1-\mu) \Phi'_{\varepsilon}(u^{\varepsilon}_{2}(t)) -\Phi'_{\varepsilon}(u^{\varepsilon}_{4}(\mu)(t)) \,, \, U^{+}_{\varepsilon}(\mu)(t) \rangle dt.
\end{eqnarray*}
As
\begin{eqnarray*}
&&< \mu \Phi'_{\varepsilon}(u^{\varepsilon}_{1})
         + (1-\mu) \Phi'_{\varepsilon}(u^{\varepsilon}_{2}) -\Phi'_{\varepsilon}(u^{\varepsilon}_{4}(\mu)) \,, \, U^{+}_{\varepsilon}(\mu) > =
\nonumber\\\nonumber\\ &=&
 \int_{\Gamma_{2}'}{q\mu u^{\varepsilon}_{1}  U^{+}_{\varepsilon}(\mu) \over \sqrt{\varepsilon^{2} + |u^{\varepsilon}_{1}|^{2}}} ds
+  \int_{\Gamma_{2}'}{q(1-\mu) u^{\varepsilon}_{2}  U^{+}_{\varepsilon}(\mu) \over \sqrt{\varepsilon^{2} +  |u^{\varepsilon}_{2}|^{2}}} ds
-\int_{\Gamma_{2}'}{qu^{\varepsilon}_{4}(\mu)  U^{+}_{\varepsilon}(\mu)  \over \sqrt{\varepsilon^{2} +  |u^{\varepsilon}_{4}|^{2}}} ds
\end{eqnarray*}
where $\Gamma_{2}'= \Gamma_{2}\cap\{u^{\varepsilon}_{4}(\mu)> u^{\varepsilon}_{3}(\mu)\}$. The function $x\mapsto \psi(x)= \dfrac{x}{\sqrt{\varepsilon^{2} + x^{2}}}$ for $x\in \br$ is increasing ($\psi'(x) = \varepsilon^{2} (\varepsilon^{2} + x^{2})^{-3\over 2}  >0$) so
\begin{eqnarray*}
&& \int_{\Gamma_{2}'}{q\mu u^{\varepsilon}_{1}  U^{+}_{\varepsilon}(\mu) \over \sqrt{\varepsilon^{2} +\|u^{\varepsilon}_{1}\|_{\br^{N}}^{2}}}  ds
+  \int_{\Gamma_{2}'}{q(1-\mu) u^{\varepsilon}_{2}  U^{+}_{\varepsilon}(\mu) \over \sqrt{\varepsilon^{2} + |u^{\varepsilon}_{2}|^{2}}} ds
-\int_{\Gamma_{2}'}{qu^{\varepsilon}_{4}(\mu)  U^{+}_{\varepsilon}(\mu)  \over \sqrt{\varepsilon^{2} +
|u^{\varepsilon}_{4}|^{2}}}ds
\nonumber\\\nonumber\\
&\leq&
 \int_{\Gamma_{2}'}{q\mu u^{\varepsilon}_{1}  U^{+}_{\varepsilon}(\mu) \over \sqrt{\varepsilon^{2} +
|u^{\varepsilon}_{1}|^{2}}} ds
+  \int_{\Gamma_{2}'}{q(1-\mu) u^{\varepsilon}_{2}  U^{+}_{\varepsilon}(\mu) \over \sqrt{\varepsilon^{2} +
|u^{\varepsilon}_{2}|^{2}}}ds
-\int_{\Gamma_{2}'}{qu^{\varepsilon}_{3}(\mu)  U^{+}_{\varepsilon}(\mu)  \over \sqrt{\varepsilon^{2} +
|u^{\varepsilon}_{3}|^{2}}} ds.
\end{eqnarray*}
Moreover the function $\psi$ is concave on $\br^{+}\setminus\{0\}$
($\psi''(x)= -3\varepsilon^{2} x (\varepsilon^{2} + x^{2})^{-5\over 2}  <0$) thus
 \begin{eqnarray}\label{eq3.12}
{1\over 2}\|U^{+}(\mu)(T)\|^{2}_{H}+ \lambda\int_{0}^{T}
\| U^{+}(\mu)(t)\|^{2}_{V} dt \leq 0.
\end{eqnarray}
As $U^{+}_{\varepsilon}(\mu)=0$ on $\{\Gamma_{2}\times[0 , T]\}\cap\{u^{\varepsilon}_{4}(\mu)\leq u^{\varepsilon}_{3}(\mu)\}$ so
\begin{eqnarray}\label{eq2.10}
u^{\varepsilon}_{4}(\mu)\leq u^{\varepsilon}_{3}(\mu) \quad \forall  \mu\in [0 , 1].
\end{eqnarray}
Now we must prove  that $u^{\varepsilon}_{3}(\mu) \to u_{3}(\mu)$ and   $u^{\varepsilon}_{4}(\mu) \to u_{4}(\mu)$ strongly in $L^{2}(0 , T ; H)$ when $\varepsilon \to 0$.
Taking in (\ref{ive}) $v= u_{b}\in K$ with $u^{\varepsilon}= u^{\varepsilon}_{i}$ ($i=1 , 2$), we deduce that
   \begin{eqnarray*}
 \langle \dot{u}^{\varepsilon}_{i}  \, ,  \, u^{\varepsilon}_{i}-u_{b}\rangle+   a( u^{\varepsilon}_{i}-u_{b} \, ,  \, u^{\varepsilon}_{i}-u_{b}) +
\langle  \Phi'_{\varepsilon}(u^{\varepsilon}_{i}) \,, \, u^{\varepsilon}_{i}
\rangle
 \leq
 a( u_{b}  \, ,  \, u_{b}-u^{\varepsilon}_{i})
\nonumber\\ +
\langle \Phi'_{\varepsilon}(u^{\varepsilon}_{i}) \,, \, u_{b}\rangle  -  \langle g_{i} \,, \, u_{b}-u^{\varepsilon}_{i} \rangle.
   \end{eqnarray*}
 As
$$
 \langle  \Phi'_{\varepsilon}(u^{\varepsilon}_{i}) \,, \, u^{\varepsilon}_{i}  \rangle
\geq 0 \quad \mbox{ and }\quad
|\langle \Phi'_{\varepsilon}(u^{\varepsilon}_{i}) \,, \, u_{b}\rangle|
\leq \int_{\Gamma_{2}}q|u_{b}|ds $$
we deduce, using the Cauchy-Schwartz inequality, that
$\|u^{\varepsilon}_{i}\|_{L^{2}(0 , T; V)}$ so  also   $\|u^{\varepsilon}_{3}(\mu)\|_{L^{2}(0 , T; V)}$
are bounded independently from $\varepsilon$. By Theorem \ref{th1} we get
\begin{eqnarray*}
{1\over 2}\|u^{\varepsilon}_{3}(\mu)-u^{\varepsilon}_{4}(\mu)\|_{L^{\infty}(0 , T; H)} + \lambda\|u^{\varepsilon}_{3}(\mu)-u^{\varepsilon}_{4}(\mu)\|_{L^{2}(0 , T; V)}\leq \mu(1-\mu)({\cal A}^{\varepsilon} (T,g_{1}) + {\cal B}^{\varepsilon}(T,g_{2}))
\nonumber\\
\leq \mu(1-\mu){1\over 2}\left(\|g_{1}-g_{2}\|^{2}_{L^{2}(0 , T; H)} + \|u^{\varepsilon}_{1}-u^{\varepsilon}_{2}\|^{2}_{L^{2}(0 , T; H)}\right)
  \quad \forall \mu\in [0 , 1],
\end{eqnarray*}
thus $\|u^{\varepsilon}_{4}(\mu)\|_{L^{2}(0 , T; V)}$ is also bounded independently from $\varepsilon$.
So there exists   $l_{i}\in V$, for $i=1, \cdots, 4$, such that
\begin{eqnarray}\label{uib}
 u^{\varepsilon}_{i}\tow l_{i}  \mbox{ in } L^{2}(0 , T; V)\mbox{ weak},
  \mbox{ and in  } L^{\infty}(0 , T; H) \mbox{ weak star}.
\end{eqnarray}
 We  check now that $l_{i}= u_{i}$. Indeed for  $i=1, 2$ or  $4$ and as $\Phi$ is convex functional we have,
\begin{eqnarray*}\label{ive1}
  \langle \dot{u}^{\varepsilon}_{i}\, , \, v -u^{\varepsilon}_{i} \rangle +  a(u^{\varepsilon}_{i}\, , \, v -u^{\varepsilon}_{i} ) + \Phi_{\varepsilon}(v) -   \Phi_{\varepsilon}(u^{\varepsilon}_{i}) \geq \qquad \qquad \qquad
\nonumber\\
\langle \dot{u}^{\varepsilon}_{i}\, , \, v -u^{\varepsilon}_{i} \rangle +  a(u^{\varepsilon}_{i}\, , \, v -u^{\varepsilon}_{i} ) +
  \langle  \Phi'_{\varepsilon}(u^{\varepsilon}_{i}) \,, \, v -u^{\varepsilon}_{i}  \rangle
 \geq   \langle g_{i} \,, \, v -u^{\varepsilon}_{i}  \rangle, \quad a.e.\,  t\in ]0 , T[
\end{eqnarray*}
thus
\begin{eqnarray}\label{i1}
  \langle \dot{u}^{\varepsilon}_{i}\, , \, v -u^{\varepsilon}_{i} \rangle +  a(u^{\varepsilon}_{i}\, , \, v -u^{\varepsilon}_{i} ) + \Phi_{\varepsilon}(v) -   \Phi_{\varepsilon}(u^{\varepsilon}_{i})
 \geq   \langle g_{i} \,, \, v -u^{\varepsilon}_{i}  \rangle, \quad a.e.\,  t\in ]0 , T[.
\end{eqnarray}
  Taking $v=u^{\varepsilon}_{i}\pm \varphi$,  in (\ref{i1}) we have
\begin{eqnarray}\label{dib}
  \langle \dot{u}^{\varepsilon}_{i}\, , \, \varphi \rangle= -
  a(u^{\varepsilon}_{i} \, , \, \varphi)  +   \langle g_{i}  \, , \, \varphi\rangle
, \quad \forall \varphi\in L^{2}(0 , T, H^{1}_{0}(\Omega)).
\end{eqnarray}
As $H^{1}_{0}(\Omega)\subset V$ with  continuous inclusion but not dense, so $V'$ (the topological  dual of the space $V$)  is not identifiable  with a subset of $H^{-1}(\Omega)$.
However, following \cite{MT} we can use the Hahn-Banach Theorem in order to extend any element in
$H^{-1}(\Omega)$ to an element of $V'$ preserving its norm.
So from (\ref{uib}) and (\ref{dib}) we conclude that
 \begin{eqnarray}\label{convui}
\left.
\begin{array}{ll}
u^{\varepsilon}_{i}\tow l_{i}  \mbox{ in } L^{2}(0 , T, V) \mbox{ weak},
\mbox{ in   } L^{\infty}(0 , T, H)
\mbox{ weak star},  \\
\mbox{ and }
 \dot{u}^{\varepsilon}_{i}\tow \dot{l}_{i}  \mbox{ in } L^{2}(0 , T, V')
 \mbox{ weak}.
\end{array}
\right\}
\end{eqnarray}
Then from (\ref{i1}), and following (\cite{DL1972, Ta1982})  we can write
\begin{eqnarray*}\label{}
\int_{0}^{T} \left\{
\langle \dot{u}^{\varepsilon}_{i}\, , \, v \rangle +  a(u^{\varepsilon}_{i}\, , \, v )
 + \Phi_{\varepsilon}(v)  -\langle g_{i} \,, \, v -u^{\varepsilon}_{i}  \rangle\right\} dt
\geq
\int_{0}^{T} \left\{
\langle \dot{u}^{\varepsilon}_{i}\, , \, u^{\varepsilon}_{i} \rangle +  a(u^{\varepsilon}_{i}\, , \, u^{\varepsilon}_{i} )
 +    \Phi_{\varepsilon}(u^{\varepsilon}_{i}) \right\} dt
\nonumber\\
={1\over 2}\|u^{\varepsilon}_{i}(T)\|^{2}_{H}- {1\over 2}\|u_{b}(T)\|^{2}_{H} +
\int_{0}^{T} \left\{a(u^{\varepsilon}_{i}\, , \, u^{\varepsilon}_{i} ) +    \Phi_{\varepsilon}(u^{\varepsilon}_{i}) \right\} dt.
\qquad
\end{eqnarray*}
Using the property of $\Phi_{\varepsilon}$ we have $\liminf_{\varepsilon\to 0} \Phi_{\varepsilon}(u^{\varepsilon}_{i})\geq \Phi(l_{i})$, and (\ref{convui}) we obtain
\begin{eqnarray}\label{int2}
\int_{0}^{T} \left\{
\langle \dot{l}_{i}\, , \, v \rangle +  a(l_{i}\, , \, v ) + \Phi(v)  -\langle g_{i} \,, \, v -l_{i}  \rangle\right\} dt
\geq
\int_{0}^{T} \left\{
\langle \dot{l}_{i}\, , \, l_{i} \rangle +  a(l_{i}\, , \, l_{i} ) +    \Phi(l_{i}) \right\} dt. \quad
\end{eqnarray}
 Let $w\in K$ and any $t_{0}\in ]0 , T[$ then we consider the open interval
${\cal O}_{j}= ]t_{0}-{1\over j} , t_{0}+{1\over j}[\subset  ]0 , T[$
for $j\in \bn^{\star}$ sufficiently large we take in (\ref{int2})
$ v=
\left\{
\begin{array}{ll}
 w \mbox{ if } t\in {\cal O}_{j},\\
 l_{i}(t)  \mbox{ if }  t\in ]0 , T[\setminus{\cal O}_{j}
\end{array}
\right.
$
to get
\begin{eqnarray}\label{int1}
\int_{{\cal O}_{j}} \left\{
\langle \dot{l}_{i}\, , \, w-  l_{i}\rangle +  a(l_{i}\, , \,   w-  l_{i}) + \Phi(w)  - \Phi(l_{i})\right\} dt \geq  \int_{{\cal O}_{j}} \langle g_{i} \,, \,  w-  l_{i}  \rangle dt.
\end{eqnarray}
We use now the Lebesgues Theorem  to  obtain, when $j\to +\infty$
\begin{eqnarray}\label{e1}
 \langle \dot{l}_{i}\, , \, w-  l_{i}\rangle + a(l_{i}\, , w-  l_{i} )+ \Phi(w) -   \Phi(l_{i})
\geq  \langle g_{i} \,, \, w -l_{i}  \rangle, \quad a.e.  \, t\in ]0 , T[.
\end{eqnarray}
So by the  uniqueness of  the  solution of the parabolic  variational inequality of second kind (\ref{eq1}), we deduce that $l_{i}= u_{i}$.

To finish the proof we check the strong convergence of $u^{\varepsilon}_{i}$ to $u_{i}$.
 Indeed  for  $i=1,2$ or  $4$ taking $v=u_{i}(t)$ in (\ref{eq1}) where $u= u^{\varepsilon}_{i}$ then
       $v= u^{\varepsilon}_{i}(t)$ in (\ref{eq1}) where $u=u_{i}$, 
then by addition, and integration over the time interval $[0 , T]$ we obtain
\begin{eqnarray}\label{2.10}
{1\over 2} \|u_{i}(T) -u^{\varepsilon}_{i}(T)\|_{H}^{2}+  \int_{0}^{T} a(u_{i}(t)-u^{\varepsilon}_{i}(t)\, , u_{i}(t)- u^{\varepsilon}_{i}(t)) dt
\nonumber\\
\leq \int_{0}^{T}\Phi_{\varepsilon}(u_{i}(t))- \Phi(u_{i}(t))+
 \Phi(u^{\varepsilon}_{i}(t)) -\Phi_{\varepsilon}(u^{\varepsilon}_{i}(t)) dt
\end{eqnarray}
as
$$\Phi_{\varepsilon}(v) - \Phi(v) =  \int_{\Gamma_{2}} q(\sqrt{\varepsilon^{2} + |v|^{2}} - |v|) ds\leq \varepsilon \sqrt{|\Gamma_{2}|}\|q\|_{L^{2}(\Gamma_{2})},$$
so from (\ref{2.10})
\begin{eqnarray*}
 {1\over 2} \|u_{i} -u^{\varepsilon}_{i}\|_{L^{\infty}(0 , T , H)}^{2}+  \int_{0}^{T} a(u_{i}(t)-u^{\varepsilon}_{i}(t)\, , u_{i}(t)- u^{\varepsilon}_{i}(t)) dt
\leq 2T\varepsilon \sqrt{|\Gamma_{2}|}\|q\|_{L^{2}(\Gamma_{2})}
\end{eqnarray*}
thus
\begin{eqnarray}\label{2.11}
u^{\varepsilon}_{i}\to   u_{i}   \mbox{  strongly  in } L^{2}(0 , T; V)\cap L^{\infty}(0 , T; H) \mbox{ for } i= 1, 2, 4
\end{eqnarray}
then also
\begin{eqnarray}\label{2.12}
u^{\varepsilon}_{3}(\mu)= \mu  u^{\varepsilon}_{1} + (1-\mu)u^{\varepsilon}_{2} \to  u_{3}   \mbox{  strongly  in }  L^{2}(0 , T; V)
 \cap L^{\infty}(0 , T; H).
\end{eqnarray}
from (\ref{eq2.10}), (\ref{2.11}) and (\ref{2.12})  we get (\ref{eq3.4}).
As the proof is given for any two control $g=g_{1}$ and $g=g_{2}$ in $L^{2}(0 , T , H)$, but for
  the same $q$, $h$, $b$ and the same initial condition  \rm{(\ref{ic1})}, so we get also (\ref{eq3.4h}).
\end{proof}

\subsection{Dependency of the solutions  on the data}\label{sub-3.2}
Note that this Subsection is not needed in the last Section.
We just would like to establish  three propositions which allow us to deduce some additional and  interesting properties on the solutions of
the variational problems  $P$ and  $P_{h}$.

\begin{proposition}\label{l2.3} Let $u_{g_{n}}$, $u_{g}$ be two solutions of Problem {\rm $P$},  with
$g=g_{n}$ and $g=g$ respectively.
Assume that $g_{n}\tow g\quad in\quad L^{2}(0 , T , H)$ (weak), we get
\begin{equation}\label{eq3.5}
u_{g_{n}}\to u_{g}\quad in \quad L^{2}(0 , T , V)\cap L^{\infty}(0 , T , H)\quad (strong)
\end{equation}
\begin{equation}\label{eq3.5x}
\dot{u}_{g_{n}}\to \dot{u}_{g}\quad in \quad L^{2}(0 , T , V')\quad (strong).
\end{equation}
Moreover
\begin{equation}\label{eq3.6}
 g_{1}\geq g_{2} \quad in \quad \Omega\times[0 , T] \quad then \quad u_{g_{1}}\geq u_{g_{2}}\quad in \quad \Omega\times[0 , T].
\end{equation}
\begin{equation}\label{eq3.7}
 u_{min(g_{1} , g_{2})} \leq u_{4}(\mu) \leq u_{max(g_{1} , g_{2})}, \qquad \forall \mu\in [0 , 1].
\end{equation}
Let $u_{g_{1}h}$, $u_{g_{2}h}$ be two solutions of Problem {\rm$P_{h}$},  with
$g=g_{1}$ and $g=g_{2}$ respectively for all $h>0$,  we get
\begin{equation}\label{eq3.6new}
 g_{1}\geq g_{2} \quad in \quad \Omega\times[0 , T] \quad then \quad u_{g_{1}h}\geq u_{g_{2}h}\quad in \quad \Omega\times[0 , T].
\end{equation}
\begin{equation}\label{eq3.7h}
 u_{min(g_{1} , g_{2})h} \leq u_{h4}(\mu)
 \leq
 u_{max(g_{1} , g_{2})h} \qquad \forall \mu\in [0 , 1].
\end{equation}
\end{proposition}
\begin{proof}
Let $g_{n}\tow g$ in $L^{2}(0 , T , H)$,  $u_{g_{n}}$ and $u_{g}$  be in $L^{2}(0 , T , K)$ such that
\begin{eqnarray}\label{gh}
 \langle \dot{u}_{g_{n}} , v-u_{g_{n}}\rangle + a(u_{g_{n}} , v-u_{g_{n}})
+\Phi(v)- \Phi(u_{g_{n}})
\geq ( g_{n} , v-u_{g_{n}})   \nonumber\\ \qquad \forall v\in K , \quad a.e.  \, t\in ]0 , T[.
\end{eqnarray}
Remark also that  $V_{2}= \{ v \in V : \quad v_{|_{\Gamma_{2}}}= 0\}\subset V$ with  continuous inclusion
 but not dense, so $V'$ is not identifiable
 with a subset of $V'_{2}$. However, following again \cite{MT} we can use the Hahn-Banach Theorem in
order to extend any element in $V'_{2}$ to an element of $V'$ preserving its norm.
So with the same arguments as in (\ref{i1})- (\ref{e1}),  we conclude that there exists $\eta$
 such that (eventually for a subsequence)
 \begin{eqnarray}\label{eqW}
\left.\begin{array}{ll}
u_{g_{n}}  \tow \eta  \mbox{ in } L^{2}(0 , T, V) \mbox{ weak},
 \mbox{ in   } L^{\infty}(0 , T, H)
\mbox{ weak star, } \\
\mbox{ and }
 \dot{u}_{g_{n}}\tow \dot{\eta} \mbox{ in } L^{2}(0 , T, V')\mbox{ weak}
 \end{array}\right\}
\end{eqnarray}
Using  (\ref{eqW})  and  taking $n\to +\infty$ in (\ref{gh}), we get
\begin{equation}\label{qh}
\langle \dot{\eta} , v-\eta\rangle + a(\eta , v-\eta ) +\Phi(v)- \Phi(u_{\eta})\geq ( g , v-\eta), \qquad \forall v\in K, \quad a.e. \, t\in ]0 , T[,
\end{equation}
by the uniqueness of the solution of (\ref{eq1}) we obtain that $\eta= u_{g}$.
Taking now $v= u_{g}(t)$ in (\ref{gh}) and $v= u_{g_{n}}(t)$ in (\ref{qh}), we get by addition
and integration over $[0 , T]$ we obtain
\begin{equation*}
{1\over 2}\|u_{g_{n}}(T) - u_{g}(T)\|^{2}_{H}  +\lambda\|u_{g_{n}} - u_{g}\|_{L^{2}(0, T , V)}^{2}\leq \int_{0}^{T}(g_{n}(t) -g(t)\, , \, u_{g_{n}}(t)- u_{g}(t))dt,
\end{equation*}
so from the above inequality and (\ref{eqW}) we deduce (\ref{eq3.5}).
To prove (\ref{eq3.6}) we take first  $v=u_{1}(t)+(u_{1}(t)-u_{2}(t))^{-}$ (which is in $K$)  in (\ref{eq1}) where $u=u_{1}$  and
$g= g_{1}$,  then taking  $v=u_{2}(t) - (u_{1}(t)-u_{2}(t))^{-}$ (which also is in $K$) in (\ref{eq1})  where $u=u_{2}$
and $g= g_{2}$,  we get
$$
{1\over 2}\|(u_{1}(T) - u_{2}(T))^{-}\|^{2}_{H}  + \lambda\|(u_{1} - u_{2})^{-}\|_{L^{2}(0, T , V)}^{2}
 \leq \int_{0}^{T}(g_{2}(t)-g_{1}(t)\, ,\,  (u_{1}(t)-u_{2}(t))^{-} )dt
$$
as
$$\Phi(u_{1}) -  \Phi(u_{1} +(u_{1}-u_{2})^{-})+ \Phi(u_{2})-  \Phi(u_{2} -(u_{1}-u_{2})^{-})=0.$$
So if $g_{2}-g_{1}\leq 0$ in $\Omega\times[0 , T]$ then
$\|(u_{1}-u_{2})^{-}\|_{L^{2}(0 , T, V)}=0$, and as
$(u_{1}-u_{2})^{-}=0$ on $\Gamma_{1}\times]0 , T[$ we have by the Poincar\'e inequality that $u_{1}-u_{2}\geq 0$ in $\Omega\times[0 , T]$. Then  (\ref{eq3.7})  follows from (\ref{eq3.6}) because
$$ min\{g_{1} , g_{2}\} \leq \mu g_{1} +(1-\mu) g_{2} \leq max\{g_{1} , g_{2}\} \quad \forall \mu \in [0 , T].
$$
Similarly taking  $v=u_{g_{1}h}(t)+(u_{g_{1}h}(t)-u_{g_{2}h}(t))^{-}$ (which is in $V$)  in (\ref{iv2}) where $u=u_{g_{1}h}$  and
$g= g_{1}h$,  then taking  $v=u_{g_{2}h}(t) - (u_{g_{1}h}(t)-u_{g_{2}h}(t))^{-}$ (which also is in $V$) in (\ref{iv2})  where $u=u_{g_{2}h}$
and $g= g_{2}h$,  we get
\begin{eqnarray*}
{1\over 2}\|(u_{g_{1}h}(T) - u_{g_{2}h}(T))^{-}\|^{2}_{H}  + \lambda\|(u_{g_{1}h} - u_{g_{2}h})^{-}\|_{L^{2}(0, T , V)}^{2}
+ h\|(u_{g_{1}h} - u_{g_{2}h})^{-}\|_{L^{2}(0, T , L^{2}(\Gamma_{1}))}^{2}\nonumber\\
 \leq \int_{0}^{T}(g_{2}(t)-g_{1}(t)\, ,\,  (u_{1}(t)-u_{2}(t))^{-} )dt
\end{eqnarray*}
 so we get  also (\ref{eq3.6new}), then  (\ref{eq3.7h}) follows.
\end{proof}
The following propositions  \rm{\ref{l2.3h}} and \rm{\ref{l2.3h2}} are to give, with   some assumptions, a first information that
 the sequence $(u_{g_{h}})_{ h> 0}$ is increasing and
  bounded, therefore it is convergent in some sense. Remark from  (\ref{pmx})  that $u_{g_{h}} \geq 0$  although $g <0$,
 provided to take the parameter $h$ sufficiently large.
 \begin{proposition}\label{l2.3h} Assume that  $h>0$ and is sufficiently large,
$b$ is a positive constant, $q\geq 0$ on $\Gamma_{2}\times[0 , T]$, then we have
\begin{equation}\label{eq3.71}
g\leq 0 \mbox{ in } \Omega\times[0 , T]  \Longrightarrow 0\leq u_{g_{h}}\leq b
 \mbox{ in  } \Omega\cup\Gamma_{1}\times[0 , T],
\end{equation}
 \end{proposition}
\begin{proof}
Taking in (\ref{iv2}) $u= u_{g_{h}}(t)$ and $v= u_{g_{h}}(t)- (u_{g_{h}}(t)-b)^{+}$,  we get
 \begin{eqnarray*}\label{}
\langle \dot{u}_{g_{h}}\, ,\,  (u_{g_{h}}-b)^{+}\rangle
   + a_{h}(u_{g_{h}} \, ,\,  (u_{g_{h}}-b)^{+})
- \Phi(u_{g_{h}}- (u_{g_{h}}-b)^{+})+ \Phi(u_{g_{h}})
\nonumber\\ \leq
(g \, ,\, (u_{g_{h}}-b)^{+})   + h\int_{\Gamma_{1}}b(u_{g_{h}}-b)^{+} ds,\quad a.e.\, t\in
]0 , T[
\end{eqnarray*}
as $b$ is constant we have
$a(b  \, ,\,  (u_{g_{h}}(t)-b)^{+}) =0$ so a.e. $t\in ]0 , T[$
\begin{eqnarray*}
{1\over 2}{\partial \over\partial t}\left(\|(u_{g_{h}}(t)-b)^{+}\|^{2}_{H}\right)
+ a((u_{g_{h}}-b)^{+} \, ,\,  (u_{g_{h}}-b)^{+})
+ h\int_{\Gamma_{1}}u_{g_{h}}(u_{g_{h}}-b)^{+} ds\nonumber\\
 \leq (g \, ,\, (u_{g_{h}}-b)^{+})
 +h\int_{\Gamma_{1}}b(u_{g_{h}}-b)^{+} ds
 +\Phi(u_{g_{h}}- (u_{g_{h}}-b)^{+})- \Phi(u_{g_{h}}),
\end{eqnarray*}
as $u_{g_{h}}(0)=b$  and
\begin{eqnarray*}
\Phi(u_{g_{h}}- (u_{g_{h}}-b)^{+})- \Phi(u_{g_{h}})
&=& \int_{\Gamma_{2}}q ( |u_{g_{h}}- (u_{g_{h}}-b)^{+}|- |u_{g_{h}}|) ds\leq 0,
\end{eqnarray*}
so
\begin{eqnarray*}
{1\over 2}\|(u_{g_{h}}(T)-b)^{+}\|^{2}_{H}
+ \int_{0}^{T}a_{h}((u_{g_{h}}(t)-b)^{+} \, ,\,  (u_{g_{h}}(t)-b)^{+}) dt
\leq  \nonumber\\
 \leq\int_{0}^{T}(g(t) \, ,\, (u_{g_{h}}(t)-b)^{+})dt \leq 0,
\end{eqnarray*}
thus (\ref{eq3.71}) holds.
\end{proof}

 \begin{proposition}\label{l2.3h2}
 Assume that  $h>0$ and is sufficiently large.
Let $g$, $g_{1}$, $g_{2}$ in $L^{2}(0 , T , H)$, $q\in L^{2}(0 , T ,  L^{2}(\Gamma_{2}))$ and $b$  is a positive constant,
we have
\begin{equation}\label{eq3.72}
 g_{2}\leq g_{1}\leq 0  \mbox{ in } \Omega\times[0 , T] \quad and
\quad h_{2}\leq h_{1}
  \Longrightarrow  0\leq u_{g_{2}h_{2}}\leq u_{g_{1} h_{1}}   \mbox{ in  } \Omega\times[0 , T],
\end{equation}
\begin{equation}\label{eq3.73}
  g \leq 0 \mbox{ in } \Omega\times[0 , T]
  \Longrightarrow 0\leq u_{g_{h}}\leq  u_{g}   \mbox{ in  } \Omega\times[0 , T], \quad \forall h>0.
\end{equation}
\begin{equation}\label{eq3.74}
h_{2}\leq h_{1}
\Longrightarrow \|u_{g_{h_{2}}}-  u_{g_{h_{1}}}\|_{L^{2}(0 , T , V)}\leq {\|\gamma_{0}||\over \lambda_{1}\min(1 , h_{2})}\|b-u_{g_{h_{1}}}\|_{L^{2}(0 , T ,  {\bf L}^{2}(\Gamma_{1}))} (h_{1}-h_{2})
\end{equation}
 \end{proposition}
 \begin{proof}
To check  (\ref{eq3.72}) we take first  $v=u_{g_{1}h_{1}}(t)+(u_{g_{2}h_{2}}(t)-u_{g_{1}h_{1}}(t))^{+}$,
for $t\in [0 , T]$, in (\ref{iv2}) where $u= u_{g_{1}h_{1}}$, $g=g_{1}h_{1}$   and $h=h_{1}$, 
 then taking
  $v= u_{g_{2}h_{2}}(t)-(u_{g_{2}h_{2}}(t)-u_{g_{1}h_{1}}(t))^{+}$  in (\ref{iv2}) where $u= u_{g_{2}h_{2}}$, $g=g_{2}h_{2}$  and $h=h_{2}$,
adding the two obtained inequalities,  as
 \begin{eqnarray*}
\Phi(u_{g_{1}h_{1}}+(u_{g_{2}h_{2}}-u_{g_{1}h_{1}})^{+})
-\Phi(u_{g_{1}h_{1}})
+\Phi(u_{g_{2}h_{2}}-(u_{g_{2}h_{2}}-u_{g_{1}h_{1}})^{+}))-\Phi(u_{g_{2}h_{2}})= 0
 \end{eqnarray*}
 we get
 \begin{eqnarray*}
-{1\over 2}{\partial \over\partial t} \left(\|(u_{g_{2}h_{2}}-u_{g_{1}h_{1}})^{+}\|^{2}_{H}\right)
 -a(u_{g_{2}h_{2}} -u_{g_{1}h_{1}}  \, ,\, (u_{g_{2}h_{2}}-u_{g_{1}h_{1}})^{+} ) \nonumber\\
+ \int_{\Gamma_{1}} (h_{1} u_{g_{1}h_{1}} -   h_{2}u_{g_{2}h_{2}}) (u_{g_{2}h_{2}}-u_{g_{1}h_{1}})^{+} ds \geq
 (g_{1} - g_{2} \, ,\, (u_{g_{2}h_{2}}-u_{g_{1}h_{1}})^{+} )
 \nonumber\\
+
 (h_{1}-h_{2}) \int_{\Gamma_{1}}b (u_{g_{2}h_{2}}-u_{g_{1}h_{1}})^{+} ds,
\quad a.e.\, t\in ]0 , T[,
 \end{eqnarray*}
 so by integration on $]0 , T[$, we deduce
 \begin{eqnarray*}
{1\over 2}\|(u_{g_{2}h_{2}}(T)-u_{g_{1}h_{1}}(T))^{+}\|^{2}_{H}+
 \int_{0}^{T} a_{h_{2}}  ((u_{g_{2}h_{2}}-u_{g_{1}h_{1}})^{+}  \, ,\,
 (u_{g_{2}h_{2}}-u_{g_{1}h_{1}}(t))^{+} ) dt \leq
\nonumber\\
  \int_{0}^{T}(g_{2} -g_{1}\, ,\, (u_{g_{2}h_{2}}(t)-u_{g_{1}h_{1}})^{+} )dt
 +(h_{1}- h_{2}) \int_{0}^{T}\int_{\Gamma_{1}}(u_{g_{1}h_{1}}-b) (u_{g_{2}h_{2}}-u_{g_{1}h_{1}})^{+} dsdt,
 \end{eqnarray*}
and from (\ref{eq3.71}) we get (\ref{eq3.72}).
To check (\ref{eq3.73}), let $W= u_{g_{h}}(t)- u_{g}(t)$, and choose, in (\ref{iv2}), $v= u_{g_{h}}(t) -W^{+}(t)$, so a.e. $t\in ]0 , T[$
 \begin{eqnarray*}
\langle \dot{u}_{g_{h}}  \, ,\, W^{+}\rangle + a_{h}(u_{g_{h}} \, ,\, W^{+}) \leq
+\Phi(u_{g_{h}} -W^{+})-\Phi(u_{g_{h}})+
(g \, ,\, W^{+})  + h\int_{\Gamma_{1}}bW^{+} ds,
 \end{eqnarray*}
as $u_{g}= b$ on $\Gamma_{1}\times[0 , T]$ we obtain a.e. $t\in ]0 , T[$
\begin{eqnarray}\label{e5.11}
\langle \dot{u}_{g_{h}} \, ,\, W^{+}\rangle + a(u_{g_{h}} \, ,\, W^{+}) + h\int_{\Gamma_{1}} |W^{+}|^{2}ds
\leq ( g \, ,\, W^{+} ) +\Phi(u_{g_{h}} -W^{+})-\Phi(u_{g_{h}}).
 \end{eqnarray}
Then we choose, in (\ref{eq1}), $v= u_{g}(t) +W^{+}(t)$, which is in $K$ because from (\ref{eq3.71}) we have $W^{+}= 0$ on $\Gamma_{1}\times[0 , T]$,  so
 \begin{eqnarray}\label{e5.12}
 \langle \dot{u}_{g} , W^{+}(t) \rangle +  a( u_{g} , W^{+}) \geq ( g \, ,\, W^{+} ) -\Phi(u_{g} +W^{+})+\Phi(u_{g}),
 \quad a.e.\, t\in ]0 , T[.
 \end{eqnarray}
So from (\ref{e5.11}) and (\ref{e5.12}) we deduce that
\begin{eqnarray*}
{1\over 2}\|W^{+}(T)\|^{2}_{H}+ \int_{0}^{T}a(W^{+} , W^{+})dt + h\int_{\Gamma_{1}} |W^{+}|^{2}ds\nonumber\\
\leq  \Phi(u_{g_{h}} -W^{+})-\Phi(u_{g_{h}})+\Phi(u_{g} +W^{+})-\Phi(u_{g})= 0.
 \end{eqnarray*}
Then (\ref{eq3.73}) holds.
To finish the proof we must check  (\ref{eq3.74}).
We choose  $v= u_{g_{h_{1}}}(t)$ in (\ref{iv2}) where $u=u_{g_{h_{2}}}(t)$, then
choosing $v= u_{g_{h_{2}}}(t)$ in (\ref{iv2}) where $u=u_{g_{h_{1}}}(t)$,
 we get
 \begin{eqnarray*}\label{e5.15}
-\langle \dot{u}_{g_{h_{2}}} - \dot{u}_{g_{h_{1}}}\, ,\, u_{g_{h_{2}}}-u_{g_{h_{1}}}\rangle
- a(u_{g_{h_{2}}} - u_{g_{h_{1}}}\, ,\, u_{g_{h_{2}}}-u_{g_{h_{1}}})
\nonumber\\
-h_{2} \int_{\Gamma_{1}} u_{g_{h_{2}}}(u_{g_{h_{2}}}-u_{g_{h_{1}}}) ds
 + h_{1}\int_{\Gamma_{1}}u_{g_{h_{1}}}(u_{g_{h_{2}}}-u_{g_{h_{1}}})ds
  \geq \nonumber\\
 -(h_{2} -h_{1})\int_{\Gamma_{1}}b(u_{g_{h_{2}}}-u_{g_{h_{1}}}) ds,
 \quad a.e.\, t\in ]0 , T[,
 \end{eqnarray*}
then
\begin{eqnarray*}\label{e5.15}
{1\over 2}\|u_{g_{h_{2}}}(T)-u_{g_{h_{1}}}(T)\|_{H}^{2} +
\int_{0}^{T}a_{h_{2}}(u_{g_{h_{2}}} - u_{g_{h_{1}}}\, ,\, u_{g_{h_{2}}}-u_{g_{h_{1}}}) dt
   \nonumber\\ \leq
 (h_{1} -h_{2})\int_{0}^{T}\int_{\Gamma_{1}}(u_{g_{h_{1}}}-b) (u_{g_{h_{2}}}-u_{g_{h_{1}}}) dsdt.
 \end{eqnarray*}
So
\begin{eqnarray*}\label{e5.15}
{1\over 2}\|u_{g_{h_{2}}}-u_{g_{h_{1}}}\|_{L^{\infty}(0 , T , H)}^{2}
+\lambda_{1}\min\{1 , h_{2}\} \|u_{g_{h_{2}}} - u_{g_{h_{1}}}\|_{L^{2}(0 , T , V)}^{2}
 \nonumber\\
\leq\|\gamma_{0}\| (h_{1} -h_{2})\|b-  u_{g_{h_{1}}}\|_{L^{2}(0 , T , {{\bf L}^{2}(\Gamma_{1}))}}
\|u_{g_{h_{2}}}-u_{g_{h_{1}}}\|_{L^{2}(0 , T , V)}
 \end{eqnarray*}
where $\gamma_{0}$ is the  trace embedding  from $V$ to $L^{2}(\Gamma_{1})$.
Thus (\ref{eq3.74}) holds.
 \end{proof}

\section{Optimal Control problems and convergence for $h\to +\infty$}\label{secOPC}
In this  section, $b$ is not constant but a given function in $L^{2}(]0 , T[\times\Gamma_{1})$. We prove first the existence and uniqueness of the solution for   the optimal control problem  associated to the parabolic variational inequalities of second kind (\rm{\ref{eq1}}),  and for   the optimal control problem  associated also to
(\rm{\ref{iv2}}), then  in  Subsection \ref{secLim} we prove (see Lemma
\rm{\ref{l6.1}} and Theorem \rm{\ref{th6.1}})   the convergence of  the state $u_{{g_{op}}_{h}h}$ and
the optimal control ${g_{op}}_{h}$,
 when the  coefficient $h$ on $\Gamma_{1}$, goes to infinity.

The existence and uniqueness of the solution to the parabolic variational inequalities of second kind
 (\ref{eq1}) and (\ref{iv2}),
with the initial condition (\ref{ic1}), allow us to consider $g\mapsto u_{g}$ and
 $g\mapsto u_{g_{h}}$ as  functions from $L^{2}(0 , T, H)$ to $L^{2}(0 , T, V)$,  for all $h > 0$.

Using the monotony property (\ref{eq3.4}) and  (\ref{eq3.4h}), established in Theorem \ref{la}, we prove in the following  that $J$ and  $J_{h}$, defined by (\ref{e4.1}) and (\ref{Jh}),  are strictly convex applications on  $L^{2}(0 , T , H)$, so \cite{JLL} there exists  a unique solution $g_{op}$ in $L^{2}(0 , T , H)$   of the Problem  \rm{(\ref{P}), and there exists also a unique solution  $g_{op_{h}}$ in $L^{2}(0 , T , H)$ of  Problem \rm{(\ref{Ph})} for all $h>0$.

\begin{theorem}\label{th4.2}
Assume the same hypotheses of Proposition \ref{r2.3.2}. Then  $J$ and  $J_{h}$, defined by (\ref{e4.1}) and (\ref{Jh}) respectively,  are strictly  convex applications on  $L^{2}(0 , T , H)$, so there exist unique solutions $g_{op}$ and
$g_{op_{h}}$  in $L^{2}(0 , T, H)$ respectively of the Problems {\rm(\ref{P})} and
{\rm(\ref{Ph})}.
\end{theorem}
\begin{proof}
Let $u=u_{g_{i}}$ and  $u_{g_{i}h}$ be respectively the solution of
the variational inequalities {\rm(\ref{eq1})}  and {\rm(\ref{iv2})}
with $g=g_{i}$   for $i=1 , 2$.
We have
\begin{eqnarray*}
\|u_{3}(\mu)\|_{L^{2}(0 , T, H)}^{2} = \mu^{2} \|u_{g_{1}}\|_{L^{2}(0 , T, H)}^{2} + (1-\mu)^{2}\|u_{g_{2}}\|_{L^{2}(0 , T, H)}^{2}
+ 2 \mu(1-\mu)(u_{g_{1}} , u_{g_{2}})
\end{eqnarray*}
then the following equalities hold
\begin{eqnarray}\label{4.9}
\|u_{3}(\mu)\|_{L^{2}(0 , T, H)}^{2} = \mu \|u_{g_{1}}\|_{L^{2}(0 , T, H)}^{2} +
(1-\mu)\|u_{g_{2}}\|_{L^{2}(0 , T, H)}^{2} \nonumber\\
-\mu(1-\mu)\|u_{g_{2}}-
u_{g_{1}}\|_{L^{2}(0 , T, H)}^{2},
\end{eqnarray}
\begin{eqnarray}\label{4.9h}
\|u_{3h}(\mu)\|_{L^{2}(0 , T, H)}^{2} = \mu \|u_{g_{1}h}\|_{L^{2}(0 , T, H)}^{2} + (1-\mu)\|u_{g_{2}h}\|_{L^{2}(0 , T, H)}^{2}
\nonumber\\
-\mu(1-\mu)\|u_{g_{2}h}- u_{g_{1}h}\|_{L^{2}(0 , T, H)}^{2}.
\end{eqnarray}
Let now $\mu\in [0 , 1]$ and   $g_{1}, g_{2} \in L^{2}(0 , T, H)$ so
\begin{eqnarray*}
\mu J(g_{1})+ (1-\mu)J(g_{2})- J(g_{3}(\mu))= {\mu \over
2}\|u_{g_{1}}\|_{L^{2}(0 , T, H)}^{2} + {(1-\mu) \over 2}\|u_{g_{2}}\|_{L^{2}(0 , T, H)}^{2}
\nonumber\\
-{1 \over 2}\|u_{4}(\mu)\|_{L^{2}(0 , T, H)}^{2} +{M \over 2}\left\{\mu
\|g_{1}\|_{L^{2}(0 , T, H)}^{2} +(1-\mu)\|g_{2}\|_{L^{2}(0 , T, H)}^{2}
-\|g_{3}(\mu)\|_{L^{2}(0 , T, H)}^{2}\right\}
\end{eqnarray*}
using (\ref{4.9}) and  $g_{3}(\mu)= \mu g_{1} + (1-\mu)g_{2}$ we obtain
\begin{eqnarray}\label{str}
\mu J(g_{1})+ (1-\mu)J(g_{2})- J(g_{3}(\mu))=
{1 \over 2}\left(\|u_{3}(\mu)\|_{L^{2}(0 , T, H)}^{2}  -\|u_{4}(\mu)\|_{L^{2}(0 , T, H)}^{2}\right)
\nonumber\\
+{1 \over 2}\mu(1-\mu) \|u_{1}-u_{2}\|_{L^{2}(0 , T, H)}^{2}
+{M \over 2}\mu(1-\mu) \|g_{1}-g_{2}\|_{L^{2}(0 , T, H)}^{2},
\end{eqnarray}
for all $\mu\in ]0 , 1[$ and  for all $g_{1}, g_{2}$ in $L^{2}(0 , T, H)$.
 From Proposition \ref{r2.3.2} we have $u_{4}(\mu) \geq 0$ in $\Omega\times [0 , T]$ for all $\mu\in [0 , 1]$,
so using the  monotony property (\ref{eq3.4}) (Theorem \ref{la})
and   we deduce
\begin{eqnarray}\label{n1}
 \|u_{4}(\mu)\|_{L^{2}(0 , T, H)}^{2}\leq \|u_{3}(\mu)\|_{L^{2}(0 , T, H)}^{2}.
\end{eqnarray}
Finally  from  {\rm(\ref{str})} the cost functional  $J$ is strictly convex, thus \cite{JLL} the
 uniqueness of the optimal control of the problem (\ref{P}) holds.

The uniqueness of the optimal control of the problem (\ref{Ph}) follows using the analogous
 inequalities (\ref{str})-(\ref{n1}) for any $h>0$.
\end{proof}

\subsection{Convergence when $h\to +\infty$}\label{secLim}
In this last subsection we study the convergence of  the state $u_{{g_{op}}_{h}h}$ and the optimal control ${g_{op}}_{h}$,
when the  coefficient $h$ on $\Gamma_{1}$, goes to infinity. For a given $g$ in $L^{2}(0 , T , H)$ we have first the
following estimate which generalizes \cite{DT, Ta1979}.

\begin{lemma}\label{l6.1}  Let $u_{g_{h}}$ be the unique solution of the parabolic variational inequality
  {\rm(\ref{iv2})}
and $u_{g}$ the unique solution of the parabolic variational inequality  {\rm(\ref{eq1})}, then
 $$u_{g_{h}}\to u_{g}\in L^{2}(0 , T , V) \mbox{ strongly  as } h\to +\infty, \qquad \forall g\in L^{2}(0 , T , H).$$
\end{lemma}
\begin{proof}
We take  $v=u_{g}(t)$ in (\ref{iv2}) where $u= u_{g_{h}}$, and recalling that $u_{g}(t)= b$ on $\Gamma_{1}\times]0 , T[$,
 taking $ u_{g_{h}}(t)- u_{g}(t)= \phi_{h}(t)$   we obtain  for $h>1$,
 a.e. $t\in ]0 , T[$
\begin{eqnarray*}
\langle \dot{\phi_{h}}, \phi_{h}\rangle +
a_{1}(\phi_{h} \, , \, \phi_{h}) + (h-1)\int_{\Gamma_{1}}  |\phi_{h}|^{2}ds
\leq
-\langle \dot{u}_{g} , \phi_{h}\rangle
-a(u_{g}, \phi_{h})
+ ( g , \phi_{h})
    +\Phi(\phi_{h}),
\end{eqnarray*}
so we deduce that
$${1\over 2}\| \phi_{h}\|^{2}_{L^{\infty}(0 , T, H)}
+  \|\phi_{h}\|^{2}_{L^{2}(0 , T , V)}+  (h-1)\|\phi_{h}\|^{2}_{L^{2}(0 , T,  L^{2}(\Gamma_{1}))}$$
is bounded for all $h>1$,
then $\|u_{g_{h}}\|_{L^{2}(0 , T , V)}\leq\|\phi_{h}\|_{L^{2}(0 , T , V)}+ \|u_{g}\|_{L^{2}(0 , T , V)}$
is also bounded for all $h>1$. So there exists $\eta\in L^{2}(0 , T , V)$
  such that
$u_{g_{h}}\tow \eta \mbox{ weakly in } L^{2}(0 , T , V)$ and
$u_{g_{h}}\to b$   strongly on  $\Gamma_{1}$ when  $h\to +\infty$  so  $\eta(0)= b$.

Let $\varphi\in L^{2}(0 , T , V_{2})$ and taking in (\ref{iv2}) where
$u= u_{g_{h}}$,
$v= u_{g_{h}}(t) \pm \varphi(t)$, we obtain
\begin{eqnarray*}
\langle \dot{u}_{g_{h}} , \varphi\rangle = -a(u_{g_{h}} , \varphi)
+(g  , \varphi)   \quad a.e. \, t\in ]0 , T[.
\end{eqnarray*}
As   $\|u_{g_{h}}\|_{L^{2}(0 , T , V)}$
is bounded for all $h>1$, we deduce that
$\|\dot{u}_{g_{h}}\|_{L^{2}(0 , T , V'_{2})}$ is also bounded for all $h>1$.
 Following the proof of Lemma 2.3, we conclude that
 \begin{eqnarray}\label{eqW2}
\left.
\begin{array}{ll}
u_{g_{h}}  \tow \eta  \mbox{ in } L^{2}(0 , T, V)\mbox{ weak,  }
 \mbox{ and in   } L^{\infty}(0 , T, H)
\mbox{ weak star,} \\
\mbox{ and }
 \dot{u}_{g_{n}}\tow \dot{\eta} \mbox{ in } L^{2}(0 , T, V') \mbox{ weak}.
\end{array}
\right\}
\end{eqnarray}

From (\ref{iv2}) and taking $v\in K$ so $v= b$ on $\Gamma_{1}$, we obtain
\begin{eqnarray*}
\langle \dot{u}_{g_{h}}, v- u_{g_{h}}\rangle +
a(u_{g_{h}} , v- u_{g_{h}}) -  h\int_{\Gamma_{1}} |u_{g_{h}}- b|^{2}ds  \geq
\nonumber\\  \Phi(u_{g_{h}}) - \Phi(v) + ( g , v - u_{g_{h}})
 \qquad \forall v\in K, \quad a.e. \, t\in ]0 , T[,
\end{eqnarray*}
then
\begin{eqnarray}\label{eq6.1}
\langle \dot{u}_{g_{h}}, v- u_{g_{h}}\rangle +
a(u_{g_{h}} , v- u_{g_{h}})  \geq
  \Phi(u_{g_{h}}) - \Phi(v) + ( g , v - u_{g_{h}})
 \quad \forall v\in K, \, a.e. \, t\in ]0 , T[.          \quad
\end{eqnarray}
So with (\ref{eqW2}) and  the same arguments as in (\ref{i1})- (\ref{e1}), we obtain
\begin{eqnarray*}
\langle \dot{\eta}, v- \eta\rangle  +
a(\eta , v- \eta)   +\Phi(v)- \Phi(\eta) \geq ( g  , v - \eta)
 \quad \forall v\in K , \quad a.e. \, t\in ]0 , T[.
\end{eqnarray*}
and $\eta(0)=b$.
Using the uniqueness of the solution of (\ref{eq1})-(\ref{ic1}) we get that $\eta= u_{g}$.

To prove the strong convergence, we  take $v= u_{g}(t)$ in (\ref{iv2})
\begin{eqnarray*}
\langle \dot{u}_{g_{h}}  , u_{g} - u_{g_{h}} \rangle  +
a_{h}(u_{g_{h}}  , u_{g} - u_{g_{h}} ) +\Phi(u_{g})- \Phi(u_{g_{h}})  \geq
( g  ,  u_{g} - u_{g_{h}} )
\nonumber\\
 + h\int_{\Gamma_{1}} b (u_{g}  - u_{g_{h}} )ds,  \quad a.e. \, t\in ]0 , T[
\end{eqnarray*}
thus as $u_{g}= b$ on $\Gamma_{1}\times]0 , T[$,
we put $u_{g_{h}} - u_{g}= \phi_{h}$, so   a.e. $t\in ]0 , T[$
\begin{eqnarray*}
\langle \dot{\phi_{h}} \, ,\,  \phi_{h}\rangle  +
a(\phi_{h} , \phi_{h})
+ h \int_{\Gamma_{1}} |\phi_{h}|^{2}ds + \Phi(u_{g_{h}})-\Phi(u_{g})\leq
\langle \dot{u}_{g} , \phi_{h}\rangle
 +
 a( u_{g} \, ,\, \phi_{h})  + ( g ,  \phi_{h}),
\end{eqnarray*}
so
\begin{eqnarray*}
&&{1\over 2}\|\phi_{h}\|_{L^{\infty}(0 , T, H)}^{2}  +
\lambda_{h}\|\phi_{h}\|_{L^{2}(0 , T , V)}^{2}
 + \Phi(u_{g_{h}})-\Phi(u_{g})\leq
-\int_{0}^{T}\langle \dot{u}_{g}(t), \phi_{h}(t)\rangle dt
\nonumber\\
&&-\int_{0}^{T} a(u_{g}(t), \phi_{h}(t)dt + \int_{0}^{T}( g(t)  , \phi_{h}(t)dt,
\end{eqnarray*}
using the weak semi-continuity of $\Phi$ and the weak convergence (\ref{eqW})  the right side
of the just above inequality tends to zero when $h\to +\infty$, then we deduce the strong convergence
of $\phi_{h}=u_{g_{h}}-u_{g}$ to $0$ in $L^{2}(0 , T , V)\cap L^{\infty}(0 , T, H) $, for all
$g\in L^{2}(0 , T , H)$.
This ends the proof.
\end{proof}

We  give now, without need to use the notion of adjoint states \cite{JLL}, the convergence
 result which  generalizes the result obtained
 in \cite{MT} for a  parabolic variational equations (see also
\cite{arada2000, belgacem2003, GT, GT2008}).

\begin{theorem}\label{th6.1}
Let $u_{{g_{op}}_{h}h}$, ${g_{op}}_{h}$  and $u_{g_{op}}$, $g_{op}$ be respectively
 the states and the optimal control defined in the problems {\rm(\ref{P})}
 and {\rm(\ref{Ph})}.
Then
\begin{eqnarray}\label{6.1}
 \lim_{h\to +\infty}\|u_{g_{op_{h}h}}-u_{g_{op}}\|_{L^{2}(0 , T , V)}&=&
 \lim_{h\to +\infty}\|u_{g_{op_{h}h}}-u_{g_{op}}\|_{L^{\infty}(0 , T , H)} ,
\nonumber\\
&=& \lim_{h\to +\infty}\|u_{g_{op_{h}h}}-u_{g_{op}}\|_{L^{2}(0 , T , L^{2}(\Gamma_{1}))}= 0,
\end{eqnarray}
\begin{equation}\label{6.2}
 \lim_{h\to +\infty}\|g_{op_{h}}-g_{op}\|_{L^{2}(0 , T , H)}= 0.
\end{equation}
\end{theorem}
\begin{proof}
We have first
\begin{eqnarray*}
 J_{h}(g_{op_{h}})= {1\over 2}\|u_{g_{op_{h}h}}\|_{L^{2}(0 , T , H)}^{2} + {M\over 2}\|g_{op_{h}}\|_{L^{2}(0 , T , H)}^{2}
\leq {1\over 2}\|u_{g_{h}}\|_{L^{2}(0 , T , H)}^{2} + {M\over 2}\|g\|_{L^{2}(0 , T , H)}^{2},
\end{eqnarray*}
for all $g\in L^{2}(0 , T , H)$,
then for $g=0\in L^{2}(0 , T , H)$ we obtain that
\begin{eqnarray}\label{e6.3}
J_{h}(g_{op_{h}})= {1\over 2}\|u_{g_{op_{h}h}}\|_{L^{2}(0 , T , H)}^{2} + {M\over 2}\|g_{op_{h}}\|_{L^{2}(0 , T , H)}^{2}\leq {1\over 2}\|u_{0_{h}}\|_{L^{2}(0 , T , H)}^{2}
\end{eqnarray}
where $u_{0_{h}}\in L^{2}(0 , T , V)$  is the solution of the following parabolic variational inequality
$$
\langle \dot{u}_{0_{h}} , v- u_{0_{h}}\rangle +
 a_{h}( u_{0_{h}} , v- u_{0_{h}}) +\Phi(v)- \Phi(u_{0_{h}})
\geq   h\int_{\Gamma_{1}} b(v- u_{0_{h}})ds, \quad a.e. \, t\in ]0 , T[
$$
for all $v\in V$ and $u_{0_{h}}(0)= u_{b}$. Taking $v= u_{b}\in K$  we get that
   $\|u_{0_{h}}- u_{b}\|_{L^{2}(0 , T , V)}$ is bounded independently of $h$, then
  $\|u_{0_{h}}\|_{L^{2}(0 , T , H)}$ is bounded independently of $h$. So we deduce with (\ref{e6.3}) that
 $\|u_{g_{op_{h}h}}\|_{L^{2}(0 , T , H)}$ and $\|g_{op_{h}}\|_{L^{2}(0 , T , H)}$ are also bounded independently
of $h$. So there exists $f$ and $\eta$ in  $L^{2}(0 , T , H)$ such that
\begin{eqnarray}\label{6.5}
 g_{op_{h}} \tow  f \quad in \quad L^{2}(0 , T , H) \quad (weak) \quad {\rm and }\quad u_{g_{op_{h}h}}\tow\eta
\quad in \quad L^{2}(0 , T , H) \quad (weak).
\end{eqnarray}
Taking  now $v=u_{g_{op}}(t)\in K$ in  (\ref{iv2}), for $t\in ]0 , T[$,  with $u= u_{g_{op_{h}}h}$
and $g=g_{op_{h}}$, we obtain
\begin{eqnarray*}
\langle \dot{u}_{g_{op_{h}}h} , u_{g_{op}} - u_{g_{op_{h}h}}\rangle +
 a_{1}( u_{g_{op_{h}}h} , u_{g_{op}} - u_{g_{op_{h}h}})
\nonumber\\
+(h-1) \int_{\Gamma_{1}}u_{g_{op_{h}h}} (u_{g_{op}}  - u_{g_{op_{h}h}} )ds
+ \Phi(u_{g_{op}})- \Phi(u{g_{op_{h}h}})
\geq
\nonumber\\
 ( g_{op_{h}}  , u_{g_{op}}  - u_{g_{op_{h}}h} )
+
h\int_{\Gamma_{1}} b (u_{g_{op}}  - u_{g_{op_{h}}h} )ds, \quad a.e. \, t\in ]0 , T[
\end{eqnarray*}
as $ u_{g_{op}}= b$ on $\Gamma_{1}\times[0 , T]$, taking
$u_{g_{op}} - u_{g_{op_{h}}h}=\phi_{h}$
 we obtain
\begin{eqnarray*}
\langle \dot{\phi_{h}} , \phi_{h}\rangle
 +a_{1}(\phi_{h} , \phi_{h})
+(h-1) \int_{\Gamma_{1}} |\phi_{h}|^{2}ds
\leq -(g_{op_{h}}  , \phi_{h} ) \nonumber\\
+ \int_{\Gamma_{2}} q  |\phi_{h}| ds
+\langle \dot{u}_{g_{op}} , \phi_{h}\rangle +
a( u_{g_{op}} ,\phi_{h}), \quad a.e. \, t\in ]0 , T[
\end{eqnarray*}
then
\begin{eqnarray*}
{1\over 2}\|\phi_{h}\|_{L^{\infty}(0 , T, H)}^{2}
+ \lambda_{1}\|\phi_{h}\|_{L^{2}(0 , T , V)}^{2}
+(h-1) \int_{0}^{T}\int_{\Gamma_{1}}|\phi_{h}(t)|^{2}ds dt
\nonumber\\ \leq
-\int_{0}^{T}(g_{op_{h}}(t)  , \phi_{h}(t)) dt
+ \int_{0}^{T}\int_{\Gamma_{2}} q  |\phi_{h}(t)|dsdt
+\int_{0}^{T}\langle \dot{u}_{g_{op}}(t) , \phi_{h}(t)\rangle dt
\nonumber\\+
\int_{0}^{T}a( u_{g_{op_{h}}h}(t) , \phi_{h}(t))dt.
\end{eqnarray*}
There exists a constant $C>$ which does not depend on $h$ such that
\begin{eqnarray*}\label{6.q}
\|\phi_{h}\|_{L^{2}(0 , T , V)}= \|u_{g_{op_{h}}h} -u_{g_{op}}\|_{L^{2}(0 , T , V)}\leq C,
\quad \|\phi_{h}\|_{L^{\infty}(0 , T, H)}\leq C
\nonumber\\
\mbox{ and  }
 (h-1)\int_{0}^{T}\int_{\Gamma_{1}} |u_{g_{op_{h}}h}- b|^{2}ds dt \leq C,
\end{eqnarray*}
then $\eta \in L^{2}(0 , T , V)$ and
\begin{eqnarray}\label{6.6}
u_{g_{op_{h}h}} \tow \eta \quad in \quad L^{2}(0 , T , V) \quad weak \mbox{ and in }  L^{\infty}(0 , T ,  H)
 \mbox{ weak star }
\end{eqnarray}
\begin{eqnarray}\label{6.7}
 u_{g_{op_{h}h}} \to  b\quad in \quad L^{2}(0 , T , L^{2}(\Gamma_{1})) \quad strong,
\end{eqnarray}
so $\eta(t)\in K$ for all $t\in [0 , T]$.
Now taking $v\in K$  in (\ref{iv2}) where $u= u_{g_{op_{h}}h}$ and $g=g_{op_{h}}$ so
 \begin{eqnarray*}\label{}
 \langle \dot{u}_{g_{op_{h}h}}  , v- u_{g_{op_{h}h}}  \rangle +
 a_{h}(u_{g_{op_{h}h}}  , v- u_{g_{op_{h}h}} )
+\Phi(v)- \Phi(u_{g_{op_{h}h}})
\geq ( g_{op_{h}}  ,  v- u_{g_{op_{h}}h} )
\nonumber\\
+ h\int_{\Gamma_{1}} b(v- u_{g_{op_{h}}h} )ds, \quad a.e. \, t\in ]0 , T[
 \end{eqnarray*}
as $v\in K$ so $v=b$ on $\Gamma_{1}$, thus we have
 \begin{eqnarray*}\label{}
 \langle \dot{u}_{g_{op_{h}h}}  , u_{g_{op_{h}h}} -v  \rangle + a(u_{g_{op_{h}h}} , u_{g_{op_{h}h}} -v) +
 h \int_{\Gamma_{1}}|u_{g_{op_{h}h}}-b|^{2}ds  + \Phi(u_{g_{op_{h}h}})
-\Phi(v)  \nonumber\\ \leq
 \langle
-( g_{op_{h}}  ,  v- u_{g_{op_{h}h}})
 \quad a.e. \, t\in ]0 , T[.
 \end{eqnarray*}
Thus
\begin{eqnarray*}\label{}
\langle \dot{u}_{g_{op_{h}h}}  , u_{g_{op_{h}h}} -v \rangle + a(u_{g_{op_{h}h} }, u_{g_{op_{h}h}} -v)
 + \Phi(u_{g_{op_{h}h}}) -\Phi(v)
\leq
-( g_{op_{h}}  , v- u_{g_{op_{h}h}}) \nonumber\\
 \quad a.e. \, t\in ]0 , T[.
 \end{eqnarray*}
Using (\ref{6.5}) and (\ref{6.6}) and the same arguments as in (\ref{i1})- (\ref{e1}),  we deduce that
\begin{eqnarray*}\label{}
 \langle \dot{\eta} , v - \eta \rangle +  a(\eta , v- \eta)
 +\Phi(v)-  \Phi(\eta)
 \geq  (f ,  v -\eta), \quad \forall v\in K, \quad a.e.  \, t\in]0 , T[,
 \end{eqnarray*}
so also by the uniqueness of the solution of (\ref{eq1}) we obtain that
\begin{eqnarray}\label{xi}
u_{f}= \eta.
 \end{eqnarray}
 We prove that  $f= g_{op}$.  Indeed we have
 \begin{eqnarray*}\label{}
  J(f)&=&{1\over 2} \|\eta \|_{L^{2}(0 , T; H)}^{2} + {M\over 2} \|f\|_{L^{2}(0 , T; H)}^{2}
\nonumber\\
&\leq& \liminf_{h\to +\infty} \left\{{1\over 2} \|u_{g_{op_{h}}h}\|_{L^{2}(0 , T; H)}^{2} + {M\over 2} \|g_{op_{h}}\|_{L^{2}(0 , T; H)}^{2} \right\}
=\liminf_{h\to +\infty} J_{h}(g_{op_{h}})
\nonumber\\
&\leq& \liminf_{h\to +\infty} J_{h}(g)
=\liminf_{h\to +\infty}
\left\{{1\over 2} \|u_{g_{h}}\|_{L^{2}(0 , T; H)}^{2} + {M\over 2} \|g\|_{L^{2}(0 , T; H)}^{2} \right\}
 \end{eqnarray*}
using now the strong convergence $u_{g_{h}}\to u_{g}$ as $h\to
+\infty,\; \forall \; g\in H$ (see Lemma \ref{l6.1}), we obtain that
 \begin{eqnarray}\label{6.9}
J(f)\leq \liminf_{h\to +\infty} J_{h}(g_{op_{h}}) \leq
 {1\over 2} \|u_{g}\|_{L^{2}(0 , T; H)}^{2} + {M\over 2} \|g\|_{L^{2}(0 , T; H)}^{2}= J(g),
\quad \forall g\in L^{2}(0 , T; H)
 \end{eqnarray}
then by the uniqueness of the optimal control problem (\ref{P}) we get
 \begin{eqnarray}\label{f}
f= g_{op}.
 \end{eqnarray}
Now we prove the  strong convergence of  $u_{g_{op_{h}}h}$ to $\eta=u_{f}$ in
$L^{2}(0 , T , V)\cap L^{\infty}(0 , T , H)\cap L^{2}(0 , T , L^{2}(\Gamma_{1}))$,
 indeed taking $v=\eta$ in (\ref{iv2}) where $u=u_{g_{op_{h}}h}$ and $g= g_{op_{h}}$,
as $\eta(t)\in K$ for $t\in [0 , T]$,  so $\eta=b$ on $\Gamma_{1}$, we obtain
we get
 \begin{eqnarray*}\label{}
\langle \dot{u}_{g_{op_{h}h}}- \dot{\eta} ,   u_{g_{op_{h}h}}-\eta   \rangle +
a_{1}( u_{g_{op_{h}h}} -\eta ,  u_{g_{op_{h}h}}- \eta )
+(h-1) \int_{\Gamma_{1}}|u_{g_{op_{h}h}} - \eta|^{2}ds
\nonumber\\+ \Phi(u_{g_{op_{h}h}})- \Phi(\eta)
\leq  (g_{op_{h}}  ,  u_{g_{op_{h}h}}-\eta )
+ \langle \dot{\eta} ,   u_{g_{op_{h}h}}-\eta   \rangle
+ a(\eta ,  u_{g_{op_{h}h}}- \eta )
 \end{eqnarray*}
thus
\begin{eqnarray*}
{1\over 2} \|u_{g_{op_{h}h}} -\eta\|_{L^{\infty}(0 , T; H)}^{2} + \lambda_{1}\|u_{g_{op_{h}h}}
-\eta\|_{L^{2}(0 , T , V)}^{2}
\nonumber\\
+\int_{0}^{T} \{\Phi(u_{g_{op_{h}h}})- \Phi(\eta)\} dt
+(h-1) \|u_{g_{op_{h}h}} - \eta\|_{L^{2}(0 , T , L^{2}(\Gamma_{1}))} ^{2}
\nonumber\\
\leq \int_{0}^{T}(g_{op_{h}}(t) , u_{g_{op_{h}h}}(t)-\eta(t))dt
+ \int_{0}^{T}\langle \dot{\eta} ,   u_{g_{op_{h}h}}-\eta   \rangle dt
\nonumber\\
+ \int_{0}^{T}a(\eta(t) , \eta(t) - u_{g_{op_{h}h}}(t))dt.
 \end{eqnarray*}
Using (\ref{6.6}) and the weak semi-continuity of $\Phi$ we deduce that
\begin{eqnarray*}\label{}
\lim_{h\to +\infty}\|u_{g_{op_{h}h}} -\eta\|_{L^{\infty}(0 , T; H)}
&=&\lim_{h\to +\infty}\|u_{g_{op_{h}}h} -\eta\|_{L^{2}(0 , T , V)} \nonumber\\
&=& \|u_{g_{op_{h}h}} - \eta\|_{L^{2}(0 , T , L^{2}(\Gamma_{1}))} = 0,
 \end{eqnarray*}
and with (\ref{xi}) and (\ref{f}) we deduce (\ref{6.1}).
As $f\in L^{2}(0 , T , H)$, then from  (\ref{6.9}) with $g=f$ and (\ref{f}) we can write
\begin{eqnarray}\label{eq6.11}
 J(f)&=& J(g_{op}) ={1\over 2}\|u_{g_{op}}\|_{L^{2}(0 , T , H)}^{2} + {M\over 2}\|g_{op}\|_{L^{2}(0 , T , H)}^{2}\nonumber\\
&\leq& \liminf_{h\to+\infty} J_{h}(g_{op_{h}})
= \liminf_{h\to+\infty} \left\{{1\over 2}\|u_{g_{op_{h}h}}\|_{L^{2}(0 , T , H)}^{2} + {M\over 2}\|g_{op_{h}}\|_{L^{2}(0 , T , H)}^{2}\right\}\nonumber\\
&\leq&  \lim_{h\to+\infty} J_{h}(g_{op}) = J((g_{op})
\end{eqnarray}
and using  the strong convergence (\ref{6.1}), we get
\begin{eqnarray}\label{eq6.13}
\lim_{h\to+\infty}\|g_{op_{h}}\|_{L^{2}(0 , T , H)}= \|g_{op}\|_{L^{2}(0 , T , H)}.
\end{eqnarray}
Finally as
\begin{eqnarray}\label{eq6.14}
\|g_{op_{h}}- g_{op}\|_{L^{2}(0 , T; H)}^{2} =
 \|g_{op_{h}}\|_{L^{2}(0 , T; H)}^{2}+ \|g_{op}\|_{L^{2}(0 , T; H)}^{2}
-2(g_{op_{h}} , g_{op})
\end{eqnarray}
and by  the first part of (\ref{6.5}) we have
$$\lim_{h\to +\infty}\left(g_{op_{h}} , g_{op}\right) = \|g_{op}\|_{L^{2}(0 , T , H)}^{2},$$
so from (\ref{eq6.13}) and (\ref{eq6.14}) we  get \mbox{(\ref{6.2})}. This ends  the proof.
\end{proof}

\noindent{\bf Acknowledgements:}
 This work was realized while the
second author was a visitor at Saint Etienne University
(France) and he is grateful to this institution for its hospitality.


\end{document}